\documentclass{amsart}
\usepackage{amssymb}
\usepackage{amsmath}
\usepackage{amsthm}
\numberwithin{equation}{section}

\theoremstyle{plain}
\newtheorem{proposition}{ Proposition}[section]
\newtheorem{lemma}[proposition]{Lemma}

\newtheorem{corollary}[proposition]{Corollary}

\theoremstyle{definition}

\newtheorem{remark}[proposition]{Remark}

\newtheorem{definition}[proposition]{Definition}

\newcommand{\mc}[1]{{\mathcal #1}}
\newcommand{\norm}[1][\cdot]{\Vert #1\Vert}
\DeclareMathOperator{\supp}{supp}
\DeclareMathOperator{\range}{range}
\begin{document}
\title{An Indecomposable and unconditionally saturated Banach space}
\author{Spiros A. Argyros}
\address{S.A. Argyros\,\,\,Department of Mathematics,
National Technical University of Athens}
\email{sargyros@math.ntua.gr}
\author{Antonis Manoussakis}
\address{A. Manoussakis\,\,\,Department of Sciences, Section of
Mathematics, Technical University of Crete}
\email{amanouss@science.tuc.gr}
\begin{abstract}
We construct an indecomposable reflexive Banach space $X_{ius}$
such that every infinite dimensional closed  subspace contains an
unconditional basic sequence. We also show that every operator
$T\in \mathcal{B}(X_{ius})$ is of the form $\lambda I+S$ with $S$
a strictly singular operator.
\end{abstract}
\maketitle
\section {Introduction}
The aim of this paper is to present a Banach space which is not
the sum of two infinite dimensional closed subspaces $Y$, $Z$ with
$Y\cap Z=\{0\}$ and every closed subspace of it contains an
unconditional basic sequence. We shall denote this space as
$X_{ius}$. W.T. Gowers' famous dichotomy, \cite{G3}, provides an
alternative description of this space. Namely $X_{ius}$ is an
Indecomposable Banach space not containing any Hereditarily
Indecomposable (H.I.) subspace. The problem of the existence of
such spaces was posed by H.P. Rosenthal  and it is stated in
\cite{G2}. The interest for such spaces arises from the
coexistence of conditional (indecomposable) and unconditional
(unconditionally saturated) structure on them. This is a free
translation of W.T.Gowers' comments before stating the problem of
the existence of such spaces in \cite{G2} (Problem 5.11). We
should mention that Indecomposable spaces which are not H.I. are
already known. For example, \cite{AF} provides reflexive H.I.
spaces $X$ such that $X^{*}$ contains an unconditional basic
sequence. The methods used in \cite{AF} do not seem to be able to
provide H.I. spaces $X$ with $X^{*}$ unconditionally saturated.

The space presented in this paper is built following ideas used
for the construction of H.I. Banach spaces. The method we follow
is an adaptation  of \cite{AD}
 constructions as they were
extended in \cite{AT}. Both are variations of the fundamental
discovery of W.T. Gowers and B. Maurey, \cite{GM}. In our case we
use as an unconditional frame a mixed Tsirelson space
$T[(\mc{A}_{n_{j}},\frac{1}{m_{j}})_{j}]$ which is a space sharing
similar properties with Th. Schlumprecht's space $S$, \cite{Sh}.
The norming set $K$ of the space $X_{ius}$ is a subset of the unit
ball of the dual of $T[(\mc{A}_{n_{j}},\frac{1}{m_{j}})_{j}]$. The
only difference that the space $X_{ius}$ has from a corresponding
construction of a H.I. space concerns the definition of the
special functionals. The key observation that changing the special
functionals one could obtain interesting non H.I. spaces is due to
W.T.Gowers and it was used for the solution of important and long
standing problems in the theory of Banach space, \cite{G}.

For the space $X_{ius}$ we need the special functionals  to be
defined such that the following geometric property holds in the
space. For every $Y=\langle e_{n}\rangle_{n\in M}$, $M\in
[\mathbb{N}]$, and $(e_{n})_{n\in \mathbb{N}}$ the natural basis
of $X_{ius}$, the quotient map $Q:X_{ius}\to X_{ius}/Y$ is
strictly singular. This is equivalent  to say that $dist(S_{Z},
S_{Y})=0$ for all $Z$ infinite dimensional subspace of $X_{ius}$.
This  property clearly holds in the case of H.I. spaces. In our
case we define the special functionals such that the
aforementioned property holds and on the other hand we have
attempted to keep the dependence inside of each special functional
as small as possible. Thus going deeper in the structure of any
subspace of $X_{ius}$ the action of the special functionals
becomes negligible, which permits us to find unconditional basic
sequences. Another property of the space $X_{ius}$ concerns the
bounded linear operators. Namely every  $T: X_{ius}\to X_{ius}$ is
of the form $T=\lambda I+S$, where $S$ is strictly singular. Thus
$X_{ius}$ is not isomorphic to any of its proper subspaces.

\section{Definition of the space $X_{ius}$}
We shall use the standard notation. Thus $c_{00}$ denotes the
linear space of all eventually zero sequences and for $x\in
c_{00}$ we denote by $\text{supp}x=\{n: x(n)\not=0\}$ and by
$\text{range}(x)$ the minimal interval of $\mathbb{N}$ containing
 $\text{supp}x$. Also for $x,y\in c_{00}$ by $x<y$ we mean that
$\max\text{supp}x<\min\text{supp}y$. We shall also use the
standard results from the theory of bases of Banach spaces as they
are described in \cite{LT}.

We choose two strictly increasing sequences  $(n_{j})_{j}$,
$(m_{j})_{j}$ of positive integers, such that
\begin{enumerate}
\item[(i)] 
$m_{1}=2$  and  $m_{j+1}= m_{j}^{5}$
\item [(ii)]
$n_{1}=4$  and $n_{j+1}=(4n_{j})^{s_{j}}$ where
$2^{s_{j}}\geq m_{j+1}^{3}$\,.
\end{enumerate}
Let $\mathbf{Q}$ be the set of scalars sequences with finite
nonempty support, rational coordinates and maximum at most $1$ in
modules. We also set
\begin{align*}
\mathbf{Q_s} = \big\{(x_1,f_1,&\ldots,x_n,f_n):\;
x_i,f_i\in\mathbf{Q},\;i=1,\ldots,n    \\
& \textrm{range}(x_{i})\cup\textrm{range}(f_{i})<
\textrm{range}(x_{i+1})\cup\textrm{range}(f_{i+1}) \;\forall i<n
 \big\}.
\end{align*}
We  consider a coding function $\sigma$ (i.e. $\sigma$ is an
injection) from  $\mathbf{Q_s}$ to the set $\{2j
:j\in\mathbb{N}\}$ such that for every $\phi=(x_{1},f_{1},\ldots,
x_{n},f_{n})\in\mathbf{Q_s}$
\begin{align}
 & \sigma (x_{1},f_{1},\ldots, x_{n-1},f_{n-1})
 <\sigma (x_{1},f_{1},\ldots, x_{n},f_{n})\label{es2} \\
 & \max\{\text{range}(x_n)\cup\text{range}(f_n)\}
 \le m_{\sigma(\phi)}^{\frac{1}{2}}  \label{es1}
\end{align}
Although $x_i,f_i$ are elements of $c_{00}$ their role in the
space $X_{ius}$  we shall define is quite different. Namely $x_i$
will be elements of the space itself and $f_i$ elements of its
dual $X_{ius}^*$. For similar reasons we shall denote the standard
basis of $c_{00}$ either by $(e_n)_n$ or $(e_n^*)_n$.
\begin{definition}\label{21}
A sequence $\phi=(x_{1},f_{1},\ldots,
x_{2k},f_{2k})\in\mathbf{Q_s}$ is said to be a {\bf special
sequence of length} $\mathbf{2k}$ provided that
\begin{equation}\label{co1}
x_{1}=\frac{1}{n_{2j}}\sum_{l=1}^{n_{2j}}e_{1,l},\qquad
f_{1}=\frac{1}{m_{2j}}\sum_{l=1}^{n_{2j}}e^{*}_{1,l},\,\,
\text{for some}\,\,j\in\mathbb{N}, \,
\text{such that}\,\,\,m^{1/2}_{2j}>2k,
\end{equation}
where $(e_{1,l})_{l=1}^{n_{2j}}$ is a subset of the standard basis
of $c_{00}$ of cardinality $n_{2j}$, and for every $1\leq i\leq
k$, setting $\phi_{i}=(x_{1},f_{1},\ldots,x_{i},f_{i})$
\begin{equation}
\norm[f_{2i}]_{\infty}\leq
   \frac{1}{m_{\sigma(\phi_{2i-1})}},\quad |f_{2i}(x_{2i})|\le
   \frac{1}{m_{\sigma(\phi_{2i-1})}},\label{co3}
 \end{equation}
\begin{equation}\label{c02}
\text{if } i<k \text{ then }
x_{2i+1}=\frac{1}{n_{\sigma(\phi_{2i})}}
\sum_{l=1}^{n_{\sigma(\phi_{2i})}}e_{2i+1,l},\quad f_{2i+1} =
\frac{1}{m_{\sigma(\phi_{2i})}} \sum_{l=1}^{n_{\sigma(\phi_{2i})}}
e^{*}_{2i+1,l},
\end{equation}
where for every $i\geq 1$,
$(e_{2i+1,l})_{l=1}^{n_{\sigma(\phi_{2i})}}$ is a subset of the
standard basis of $c_{00}$ of cardinality $n_{\sigma(\phi_{2i})}$.
\end{definition}

\textbf{The norming set of the space $X_{ius}$.}

The norming set $K$ will be equal to the union
$\cup_{n=0}^{\infty}K_{n}$ and the sequence $(K_{n})_{n}$ is
increasing and inductively defined. The inductive definition of
$K_{n}$ goes as follows:

We set  $$ K_{0}^{0}=K_{0}=\{\pm
e_{n}^*:n\in\mathbb{N}\}\,\,\text{and}\,\,\,
K^{j}_{0}=\emptyset\,\,\text{for}\,\,j=1,2,\ldots\, .
$$
Assume that $K_{n-1}=\cup_{j}K_{n-1}^{j}$ has been defined.
Then we set,

(a) for  $j\in\mathbb{N}$ $$ K_{n}^{2j}=K_{n-1}^{2j}\cup
\{\frac{1}{m_{2j}}\sum_{i=1}^{d}f_{i}: d\leq
n_{2j},\,f_{1}<\ldots<f_{d},\, f_{i}\in K_{n-1}\}\,. $$

(b) For $j\in\mathbb{N}$ and every
$\phi=(x_{1},f_{1},\ldots,x_{n_{2j+1}},f_{n_{2j+1}})$ special
sequence  of length $n_{2j+1}$, (see Definition~\ref{21}), such
that $f_{2i}\in K_{n-1}^{\sigma(\phi_{2i-1})}$ for
$i=1,\ldots,n_{2j+1}/2$ (where
$\phi_{2i-1}=(x_1,f_1,\ldots,x_{2i-1},f_{2i-1})$) we define the
set
\begin{align}\label{ek0}
K_{n,\phi}^{2j+1}=\Bigl\{\frac{\pm 1}{m_{2j+1}}&
E(\lambda_{f^{\prime}_{2}}f_{1}+f^{\prime}_{2}+\ldots+
\lambda_{f^{\prime}_{n_{2j+1}}}f_{n_{2j+1}-1}+f^{\prime}_{n_{2j+1}})
\,:
\\
 &E\,\,\text{interval of}\,\,\mathbb{N},\,\,
 \text{supp}f^{\prime}_{2i}=\text{supp}f_{2i},\,\,
 f^{\prime}_{2i}\in K_{n-1}^{\sigma(\phi_{2i-1})},\,\,
 \\
& |g(x_{2i})|\le\frac{1}{m_{\sigma(\phi_{2i-1})}}\,\,\textrm{for all}\,\,g\in K_{n-1}^{\sigma(\phi_{2i-1})}
 \notag \\
 & \lambda_{f^{\prime}_{2i}}=
 f^{\prime}_{2i}(m_{\sigma(\phi_{2i-1})}x_{2i})\,\,\, \,
 \text{if}\,\,f^{\prime}_{2i}(x_{2i})\not=
 0,\,\,\,\,\, \frac{\pm 1}{n^{2}_{2j+1}}\,\,
 \text{otherwise}\Bigr\}\,.\notag
\end{align}
We define
\begin{align*}\label{ek1}
K_{n}^{2j+1}=\cup\{K_{n,\phi}^{2j+1}: \phi\,\,\text{is a
 special sequence of length $n_{2j+1}$}\}\cup
K_{n-1}^{2j+1}\,,
\end{align*}
and finally we set
$$
K_{n}=\cup_{j}K_{n}^{j}\,.
$$
\noindent This completes the inductive definition of $K_{n}$ and
we set, $$ K=\cup_{n}K_{n}\,. $$

Let us observe that the set $K$ satisfies the following properties
\begin{enumerate}
\renewcommand{\labelenumi}{(\roman{enumi})}
\item[(i)] It is symmetric and for each
$f\in K$, $\|f\|_{\infty}\le 1$.
\item[(ii)] It is closed under interval projections (i.e.
it is closed in the
restriction of its elements on intervals).
\item[(iii)] It is closed under the
$(\mathcal{A}_{n_{2j}},\frac{1}{m_{2j}})$ operations (i.e. for
$f_1<f_2<\cdots<f_d$ in $K$ with $d\le n_{2j}$ we have that
$\frac{1}{m_{2j}}\sum\limits_{l=1}^df_l\in K)$.
\item[(iv)] If $f\in K$ then either
$f=\pm e_n^*$ or $f\in K_n^j$ for $n\ge 1$, $j\in \mathbb{N}$. In
the later case we define the {\bf weight} of $f$ as $w(f)=m_j$.
Note that $w(f)$ is not necessarily unique.
 \end{enumerate}
The space $X_{ius}$ is the completion of the space
$(c_{00},\norm[\cdot]_K)$ where
$$ \norm[x]_K=\sup\{\langle f,x\rangle :
 f\in K\}\,.
$$
From the definition of the norming set $K$ it follows easily that
$(e_{n})_n$ is a bimonotone basis of $X_{ius}$. Also it is easy to
see, using (iii), that the basis $(e_n)_n$ is boundedly complete.
Indeed, for $x\in c_{00}$ and $E_1<E_2<\cdots<E_{n_{2j}}$
intervals of $\mathbb{N}$ it follows from property (iii) of the
norming set that,
$$
\norm[x]\geq \frac{1}{m_{2j}}\sum_{i=1}^{n_{2j}}\norm[E_i x]\,\,.
$$
Also from the choice of the sequences $(n_{i})_{i}$, $(m_{i})_{i}$
it follows that $\frac{n_{2j}}{m_{2j}}$ increases to infinity.
These observations easily  yield that the basis is boundedly
complete.

To prove that the space $X_{ius}$ is reflexive we need to show
that the basis is shrinking. This requires some further work and
we will present the argument later.
\begin{lemma} Let
 $\phi=(x_{1},f_{1},\ldots,x_{n_{2j+1}},f_{n_{2j+1}})$
 be a special sequence of length $n_{2j+1}$ such that:
 \begin{enumerate}
\item[(a)] $\{f_i:\; i=1,\ldots,n_{2j+1}\}\subset K$ and for
$i\ge 2$,
$w(f_i)=m_{\sigma(\phi_{i-1})}$.
\item[(b)] For $1\le i\le n_{2j+1}/2$,
$\norm[w(f_{2i})x_{2i}]\leq 1.$
\end{enumerate}
Then there exists $n\in\mathbb{N}$ such that $K_{n,\phi}^{2j+1}$
is
 nonempty.
 \end{lemma}

\textit{Notation.} For every $\phi$ special sequence of length
$n_{2j+1}$ such  that $K_{n,\phi}^{2j+1}\neq \emptyset$ for some
$n$ we define $K_{\phi}=\cup_{n}K_{n,\phi}^{2j+1}$.
\begin{remark}
Let us point out that in the definition of the special sequences
we have attempted to connect averages of the basis with block
vectors that are quite  freely chosen. This will be used to show
that the quotient map from the space to the space $X_{ius}/\langle
e_{n}\rangle_{n\in M}$ is a strictly singular operator. Moreover
we keep the dependence only between $f_{2i-1}$ and the family
$\{g\in K:\;w(g)=w(f_{2i}),\;\supp(g)=\supp(f_{2i})\}$ to ensure
that the space $X_{ius}$ is unconditionally saturated.
\end{remark}
\begin{definition}[The tree $\mathbf{\mathcal{T}}_{f}$
 of a functional $f\in K$] \label{24}
Let $f\in K$. We call \emph{tree} of $f$ (or tree corresponding to
the analysis of $f$) every finite family $\mathcal{T}_{f}
   =(f_{\alpha})_{\alpha\in\mc{A}}$ indexed by a
 finite tree $\mc{A}$ with a unique
 root $0\in\mc{A}$ such that the following conditions are
 satisfied:

1) $f_0=f$ and $f_{\alpha}\in K$ for each $\alpha\in \mathcal{A}$.

2) If $\alpha\in\mc{A}$ is terminal node then
$f_{\alpha}\in K_{0}$.

3) For every $\alpha\in \mc{A}$ which is not terminal,
denoting by $S_{\alpha}$ the set of the
immediate successors of $\alpha$,
exclusively one of the following two holds:
 \begin{enumerate}
 \item[(a)] $S_{\alpha}=\{\beta_1,\ldots,\beta_d\}$ with
  $f_{\beta_1}<\cdots<f_{\beta_d}$
 and there exists $j\in\mathbb{N}$ such that
   $d\leq n_{2j}$, and
   $f_{\alpha}=\frac{1}{m_{2j}}\sum\limits_{i=1}^d f_{\beta_i}$.
 \item[(b)] There exists a special sequence
$\phi=(x_{1},f_{1}\ldots,x_{n_{2j+1}},f_{n_{2j+1}})$ of length
$n_{2j+1}$, an interval $E$ and $\varepsilon\in\{-1,1\}$ such that
$f_{\alpha}=\frac{\varepsilon}{m_{2j+1}}
 \sum\limits_{i=1}^{n_{2j+1}/2}
 E(\lambda_{f^{\prime}_{2i}}f_{2i-1}+f^{\prime}_{2i})\in K_{\phi}$
and $\{f_\beta:\;\beta\in S_{\alpha}\}=
\{Ef_{2i-1}:\;Ef_{2i-1}\neq 0\}\cup
 \{Ef_{2i}^{\prime}:\;Ef_{2i}^{\prime}\neq 0\}$.
\end{enumerate}
\end{definition}
It follows from the inductive definition of $K$ that every $f\in
K$ admits a tree, not necessarily unique.

\section{The space $X_{ius}$ is unconditionally saturated}
This section is devoted to show that the space $X_{ius}$ is
unconditionally saturated. We start with the following: We set
$$
\widetilde{K}=\{\pm e_{n},\,\frac{1}{m_{2j}} \sum_{i\in F}\pm
e_{i}: \# F\leq n_{2j},\, j\in\mathbb{N}\}\cup\{0\}\,\,.
$$
Clearly $\widetilde{K}$ is a subset of the norming set $K$ and it
is easily checked that $\widetilde{K}$ is a countable and compact
set (in the pointwise topology). It is well known that the space
$C(\widetilde{K})$ is $c_{0}-$saturated. Observe also that
$\norm_{\widetilde{K}}\leq \norm_{X_{ius}}$ and hence  the
identity operator
$$
I:(c_{00},\norm_{X_{ius}})\to (c_{00},\norm_{\widetilde{K}})
$$
is bounded. Since the basis $(e_{n})_{n}$ of $X_{ius}$ is
boundedly complete, the space $X_{ius}$ does not contains $c_{0}$,
therefore the operator $I$ is also strictly singular. These
observations yield that every block subspace $Y$ of $X_{ius}$
contains a further block sequence $(y_{n})$ such that
$\norm[y_{n}]_{X_{ius}}=1$ and $\norm[y_{n}]_{\widetilde{K}}
\stackrel{n}{\longrightarrow}0$. Our intention is to show the
following:
\begin{proposition}\label{31}
Let $(x_{\ell})_{\ell}$ be a normalized block sequence in
$X_{ius}$ such that $\norm[x_{\ell}]_{\widetilde{K}}\to 0$. Then
there exists a subsequence $(x_{\ell})_{\ell\in M}$ of
$(x_{\ell})$ which is an unconditional basic sequence.
\end{proposition}
The proof of this proposition requires certain steps and we
attempt a sketch of the main ideas. First we assume, passing to a
subsequence, that $\norm[x_\ell]_{\widetilde{K}}<\sigma_\ell$ with
 $\sum\sigma_\ell<\frac{1}{8}$ and we claim that
 $(x_\ell)_{\ell\in\mathbb{N}}$ is an unconditional basic sequence.
Indeed, consider a norm one combination
$\sum\limits_{\ell=1}^db_\ell x_\ell$ and let
$(\varepsilon_\ell)_{\ell=1}^d\in\{-1,1\}^d$. We shall show that
$\norm[\sum\limits_{\ell=1}^d\varepsilon_\ell b_\ell
x_\ell]>\frac{1}{4}$. Choose any $f\in K$ with
$f(\sum\limits_{\ell=1}^db_\ell x_\ell)>\frac{3}{4}$ and we are
seeking a $g\in K$ such that
$g(\sum\limits_{\ell=1}^d\varepsilon_\ell b_\ell
x_\ell)\ge\frac{1}{4}$. To find such a $g$ a normal procedure is
to consider a tree $(f_\alpha)_{\alpha\in\mathcal{A}}$ of the
functional $f$ and then inductively to produce a functional $g$
with a tree $(g_\alpha)_{\alpha\in\mathcal{A}}$
 such that
 \begin{equation}     \label{sk}
 |f(x_\ell)-g(\varepsilon_{\ell}x_{\ell})|<2\sigma_{\ell}
 \end{equation}
 which easily yields the desired result.

In most of the cases, the choice for producing $g_{\alpha}$ from
$f_{\alpha}$ is straightforward. Essentially there exists only one
case where we need to be careful. That is when $f_{\alpha}\in
K_{\phi}$ for some special sequence $\phi$. (i.e.
$f_{\alpha}=\frac{\pm 1}{m_{2j+1}}E(\lambda_{f_2^{\prime}}f_1
+f_2^{\prime}+\cdots+\lambda_{f^{\prime}_{n_{2j+1}-1}}f_{n_{2j+1}-1}
+f_{n_{2j+1}}))$ and for some $i\le n_{2j+1}/2$ and $\ell<d$ we
have
\[ \max\supp x_{\ell-1}< \min\supp (f_{2i-1}) \le\max\supp x_\ell\]
\[ \max\supp f_{2i}^{\prime} \ge \min\supp x_{\ell+1} .\]
In this case we produce $g_{\alpha}$ from $f_{\alpha}$ such that
$g_{\alpha}\in K_{\phi}$. The form of $f_{\alpha}$ and hence
$g_{\alpha}$ permits us to show that
$|f_{\alpha}(x_\ell)-g_{\alpha}(\varepsilon_{\ell}x_{\ell})|<2
\sigma_{\ell}$.

 We pass now to present the proof and we start with the next
notation and definitions.

\textit{Notation.} Let $f\in K$ and
$(f_{\alpha})_{\alpha\in\mathcal{A}}$ a tree of $f$. Then for
every non terminal node $\alpha\in\mathcal{A}$ we order
 the set $S_{\alpha}$ following the natural order of
 $\{\supp f_\beta\}_{\beta\in S_{\alpha}}$.
Hence for $\beta\in S_{\alpha}$ we denote by $\beta^+$ the
immediate successor of $\beta$ in the above order if such an
object exists.

\begin{definition}
Let $f\in K$ and $(f_{\alpha})_{\alpha\in\mc{A}}$ be a tree of
$f$. A couple of functionals $f_{\alpha}$, $f_{\alpha^{+}}$ is
said to be a {\bf depended couple with respect to}
 $\mathbf{f}$,
 (w.r.t. $f$), if there exists
 $\beta\in\mc{A}$ such that $\alpha,\alpha^{+}\in S_{\beta}$,
 $f_{\beta}=\frac{\varepsilon}{m_{2j+1}}
 E(\sum\limits_{i=1}^{n_{2j+1}/2}
 \lambda_{f_{2i}^{\beta}}f_{2i-1}^{\beta}+f_{2i}^{\beta})$,\,
  $f_{\alpha}=Ef_{2i-1}^{\beta}$ and
 $f_{\alpha^{+}}=Ef_{2i}^{\beta}$ for some $i\le n_{2j+1}/2$.
 \end{definition}
\begin{definition}
 Let $(x_{k})_{k}$ be a normalized block sequence, $f\in K$ and
 $\mathcal{T}_{f}=(f_{\alpha})_{\alpha\in\mc{A}}$ be a tree of $f$.
For $k\in \mathbb{N}$, a couple of functionals $f_{\alpha}$,
$f_{\alpha^{+}}$ is said to be {\bf depended couple with respect
to} $\mathbf{f}$ {\bf and }$\mathbf{x_{k}}$ (w.r.t.) if
$f_{\alpha}$, $f_{\alpha^{+}}$ is a depended couple w.r.t. $f$ and
moreover
$$
\max\supp x_{k-1}<   \min\supp f_{\alpha}  \le\max\supp x_k
$$
$$
\,\,\,\text{and}\,\,\,
 \max\supp f_{\alpha^{+}}\geq\min\supp x_{k+1}.
 $$

 We also set
\begin{equation}\label{s2e1}
\mathcal{F}_{f,x_{k}}= \{\alpha\in\mc{A}: f_{\alpha},
f_{\alpha^{+}}\,\,\,\text{is a depended couple w.r.t.}
\,\,f\,\text{and}\,\,x_{k}\}\,.
 \end{equation}
 and
 \begin{equation}\label{100}
 \mathcal{F}_{f}=\bigcup\limits_k \mathcal{F}_{f,x_{k}}\,.
 \end{equation}
 \end{definition}
 \begin{remark}\label{34}
 Let $(x_{k})$ be a  block sequence in $X_{ius}$, $f\in K$ and
 $(f_{\alpha})_{\alpha\in\mc{A}}$ be a tree of $f$.

 1. It is easy to see that
 for every $k\in\mathbb{N}$ and every non terminal node $\alpha\in\mc{A}$
the set $S_{\alpha}\cap \mathcal{F}_{f,x_{k}}$  has at most one element.

 2. As consequence of this, we obtain that for every $k$
  and $\alpha_{1},\alpha_{2}\in\mathcal{F}_{f,x_{k}}$ with
   $\alpha_{1}\not=\alpha_{2}$ we have that
 $\alpha_1,\alpha_2$ are incomparable and
  $\vert\alpha_{1}\vert\not=\vert\alpha_{2}\vert$,
  where we denote by $\vert\alpha\vert$ the order of $\alpha$ as
  a member of the finite tree $\mc{A}$.

 3. It is also easy to see that for
$\alpha_1,\alpha_2\in \mathcal{F}_f$ with $\alpha_{1}\neq
\alpha_2$, $\alpha_1,\alpha_2$ are incomparable and hence
$\text{range}
(f_{\alpha_{1}})\cap\text{range}(f_{\alpha_{2}})=\emptyset$.
\end{remark}
\begin{lemma}\label{l1.5}
Let $(x_{k})_{k}$ be a block sequence in $X_{ius}$ such that
$\norm[x_{k}]_{\widetilde{K}}\leq\sigma_{k}$, $f\in K$ and
$(f_{\alpha})_{\alpha\in\mc{A}}$ be a tree of $f$. We set $y_{k}=
x_{k}|_{\cup_{\alpha\in \mathcal{F}_f}\supp(f_{\alpha})}$. Then we
have that
\begin{equation}\label{101}
 \vert f(y_{k})\vert\leq 2\sigma_{k}\,\,.
 \end{equation}
\end{lemma}
\begin{proof} Let us first observe that for each $q\in\mathbb{N}$ the set
 $\{\text{range} (f_{\alpha}):\; |\alpha|=q\}$ consists of pairwise disjoint sets.
 Therefore from the preceding remark we obtain that for each $k$ and each
 $q$ the set
 \[ \{\alpha\in \mathcal{F}_f:\; |\alpha|=q,\;\text{range}(f_{\alpha})\cap \text{range}(x_k)
 \neq \emptyset\}    \]
 contains at most two elements
 (one of them
 belongs to $\mathcal{F}_{f,x_{k}}$ and the
 other to $\mathcal{F}_{f,x_{\ell}}$ for some $\ell\leq
 k-1$). Therefore
\begin{eqnarray*} \vert f(y_{k})\vert & \leq &
\sum_{\alpha\in\mathcal{F}_{f}} \bigl(
\prod_{0\preceq\gamma\prec\alpha}\frac{1}{w(f_{\gamma})}
\bigr)\vert f_{\alpha}(x_{k})\vert
\\
& = &
 \sum\limits_i\sum\limits_{\alpha\in \mathcal{F}_f\,, |\alpha|=i}
\bigl(\prod_{0\preceq\gamma\prec\alpha}\frac{1}{w(f_{\gamma})}
 \bigr)\vert f_{\alpha}(x_{k})\vert\leq
 2\sigma_{k}\sum_{i}\frac{1}{m^{i}_{1}}\leq 2\sigma_{k}\,\,.
 \end{eqnarray*}
 \end{proof}
 The following lemma is the crucial step for the proof of the main
 result of this section.
 \begin{lemma}\label{l1.6}
 Let $(x_{k})_{k}$ be a block sequence in $X_{ius}$, $f\in K$ and
 $(f_{\alpha})_{\alpha\in\mc{A}}$ be a tree of $f$. For every
 $k\in\mathbb{N}$ we
  set $y_{k}=
 x_{k}|_{\cup_{\alpha\in\mathcal{F}_{f}}
 \supp(f_{\alpha})}$. Then for every choice  of signs $(\varepsilon_{k})_{k}$
 there exists a functional $g\in K$ with a tree
 $(g_{\alpha})_{\alpha\in\mc{A}}$ such that
\begin{enumerate}
\item $f(x_{k}-y_{k})=g(\varepsilon_{k}(x_{k}-y_{k}))$
\item For every $\alpha\in\mathcal{A}$, $\supp(f_{\alpha})= \supp(g_{\alpha})$
\item $\mathcal{F}_{f,x_{k}}=\mathcal{F}_{g,x_{k}}$
\end{enumerate}
for every $k=1,2,\ldots$.
\end{lemma}
\begin{proof}
For the given tree $(f_{\alpha})_{\alpha\in\mc{A}}$ of $f$, we define
\begin{align*}
D=\{\beta\in\mc{A}:\, &
\text{range}(f_{\beta})\cap\text{range}(x_{k})\not= \emptyset\,\,\text{for
at most one $k$}
\\
&\text{and if}\,\,\beta\in S_{\alpha}\,\text{then
range}(f_{\alpha})\cap\text{range}(x_{i})\not=\emptyset\,\,\text{for at
least two}\,\,\,x_{i}\}\,.
\end{align*}
 Let us observe that for every branch $b$ of $\mc{A}$, $b\cap
 D$ is a singleton. Furthermore, for $\beta\in D$ and $\gamma\in\mathcal{A}$
 with $\beta\prec \gamma$ we have that $\gamma\not\in \mathcal{F}_f$.

 The definition of $(g_\alpha)_{\alpha\in\mathcal{A}}$ requires the following three
 steps. \\
 \textit{ Step 1.} First we define the set $\{g_\beta:\; \beta\in D\}$
 as follows.

  (a)\; If $\beta\in D$ and there exists $\alpha\in \mathcal{A}$ with
  $\alpha\preceq \beta$ and $f_{\alpha},f_{\alpha^+}$ is a depended couple w.r.t. $f$
  we set $g_\beta=f_\beta$.

  (b)\; If  $\beta\in D$  does not
  belong to the previous case  and there exists a (unique)
  $k$ such that $\range(f_\beta)\cap\range(x_k)\neq \emptyset$
  then we set  $g_\beta=\varepsilon_k f_\beta$.

  (c)\;  If $\beta\in D$  does not belong to case (a) and
  $\range(f_{\beta})\cap\range(x_k)=\emptyset$ for all $k$
  then we set $g_\beta=\varepsilon_k f_\beta$ where
  \[ k=\max\{l:\; \range(x_l)<\range (f_\beta)\}.  \]
  (We have assumed that $\min\range(x_1)\le \min\range (f)$.)

  Let us comment the case (a) in the above definition. First we observe that
  the unique $\alpha\in\mathcal{A}$
   witnessing that $\beta$ belongs to the case (a) satisfies the following:
  either $\alpha=\beta$ or $|\alpha|=|\beta|-1$.
    Moreover if this $\alpha$ does not belong to $\mathcal{F}_f$ then
  $\alpha=\beta$, $\alpha^{+}\in D$.
    In this case, if we assume that
     there exists a (unique) $k$ such that
  $\range(f_{\alpha})\cap\range(x_k)\neq \emptyset$
  then $g_{\alpha^{+}}$ is defined by cases (b) or (c) and
  $g_{\alpha^{+}}=\varepsilon_k f_{\alpha^{+}}$ for the specific $k$.
  All these are straightforward
  consequences of the corresponding definitions.\\
  \textit{ Step 2.}  We set
  \[ D^+=\{\gamma\in\mathcal{A}:\; \mbox{ there exists } \beta\in D
    \mbox{ with } \beta\prec \gamma  \}.   \]
  For $\gamma \in D^+$ we set $g_{\gamma}=\varepsilon_{\beta}f_{\gamma}$
  where $\beta$ is the unique element of $D$ with $\beta\prec\gamma$
  and $\varepsilon_\beta\in\{-1,1\}$ is such that
  $g_{\beta}=\varepsilon_{\beta}f_{\beta}$.

  Clearly for every $\beta\in D\cup D^+$,
  $(g_{\gamma})_{\beta\preceq\gamma}$ is a tree of the functional
  $g_{\beta}$.  Furthermore for
   $\alpha\in D\cup D^+$ the following properties
  hold:

 \begin{enumerate}
 \item[(1)] $\text{supp}(f_{\alpha})=\text{supp}(g_{\alpha})$
 \item[(2)] $w(f_{\alpha})=w(g_{\alpha})$
 \end{enumerate}
  \textit{ Step 3.}  We set
$$
D^{-}=\{\alpha\in\mathcal{A}:\; \mbox{ there exists } \beta\in D
    \mbox{ with } \alpha\prec \beta  \}.
$$
  Observe that $\mathcal{A}=D\cup D^+\cup D^-$ and using backward induction,
  for all $\alpha\in D^-$ we shall define $g_{\alpha}$ such that the above
  (1) and (2) hold and additionally the following two properties will be
  established.

 \begin{enumerate}
  \item[(3)] For $\alpha\in D^-$, $f_{\alpha}(x_{k}-y_{k})=
 g_{\alpha}(\varepsilon_{k}(x_{k}-y_{k}))$ for all $k$.
  \item[(4)] For $\alpha\in D^-$ and each $k$ we have that
    $\mathcal{F}_{f_{\alpha},x_{k}}= \mathcal{F}_{g_{\alpha},x_{k}}$.
  \end{enumerate}

 Observe that for every $\alpha\in D^-$  we
 have that $f_{\alpha}\not\in K_{0}$ and furthermore for every $\beta\in D$
 $\mathcal{F}_{f_\beta}=\emptyset$.

 We pass now to construct inductively $g_{\alpha}$, $\alpha\in D^-$ and to establish
 properties (1)--(4). Let assume that $\alpha\in D^-$ and for every
 $\beta\in S_{\alpha}$ either $\beta\in D$ or $g_\beta$ has been defined
 and properties (1)--(4) have been established. We consider the following
 three cases.

 \textit{Case 1.\;} $w(f_{\alpha})=m_{2j}$ and $\alpha\in \mathcal{F}_f$.\\ That means
 that $f_{\alpha}=\frac{1}{m_{2j}}\sum\limits_{\beta\in S_{\alpha}}f_{\beta}$ and
 each $f_\beta=e_{\ell}^*$ for some $\ell\in\mathbb{N}$.
 Then $S_{\alpha}\subset D$ and from Step 1(a) we conclude that
 $g_{\beta}=f_{\beta}$ for all $\beta\in S_{\alpha}$. We set
$$
g_{\alpha}=
\frac{1}{m_{2j}}\sum\limits_{\beta\in S_{\alpha}}g_{\beta}=
f_{\alpha}.
$$
 Furthermore for each $k$ we have that $\supp(g_{\alpha})\cap
\supp(x_k)\subset \supp(y_k)$. Hence
\[g_{\alpha}(\varepsilon_k(x_k-y_k))=f_{\alpha}(x_k-y_k)=0\]
and also
$\mathcal{F}_{g_{\alpha}}=\mathcal{F}_{f_{\alpha}}=\emptyset$.
Thus  properties (3) and (4) hold while (1) and (2) are obvious.

 Before passing to the next case let us notice that there is no
$\alpha\in D^-$ such that $f_{\alpha},f_{\alpha^{+}}$ is a
depended couple w.r.t. $f$ and $\alpha\not\in\mathcal{F}_{f}$.
(See the comments after Step 1.)

\textit{Case 2.\;} $w(f_{\alpha})=m_{2j}$ and $\alpha\not\in \mathcal{F}_f$.\\
 From the previous observation we obtain that $\alpha\neq \beta$ for each
 $\beta\in\mathcal{A}$ with $f_{\beta},f_{\beta^+}$
 depended couple w.r.t. $f$, and we set
$$
g_{\alpha}=\frac{1}{m_{2j}}\sum\limits_{\beta\in S_{\alpha}}g_\beta.
$$
 Our inductive assumptions yield properties (1) and (2).
 To establish property (3) let $k\in\mathbb{N}$ and $\beta\in D\cap S_{\alpha}$
be such that $\range(x_k)\cap \range(f_\beta)\neq \emptyset$.
 Then $g_\beta=\varepsilon_k f_{\beta}$ hence
 \[g_{\beta}(\varepsilon_{k}(x_{k}-y_{k}))
=\varepsilon_{k}g_{\beta}(x_{k}-y_{k})=f_{\beta}(x_{k}-y_{k}).\]
 If $\beta\in D^-\cap S_{\alpha}$
 by the inductive assumption for each $k$ we have
\[g_{\beta}(\varepsilon_{k}(x_{k}-y_{k}))=f_{\beta}(x_{k}-y_{k}).\]
 Therefore \[g_{\alpha}(\varepsilon_{k}(x_{k}-y_{k}))=f_{\alpha}(x_{k}-y_{k}).\]
 Finally, for each $k$
 \[ \mathcal{F}_{f_{\alpha},x_k}
 =\bigcup\limits_{\beta\in S_{\alpha}}\mathcal{F}_{f_{\beta},x_k}
 =\bigcup\limits_{\beta\in S_{\alpha}\cap D^-}\mathcal{F}_{f_{\beta},x_k}
 =\bigcup\limits_{\beta\in S_{\alpha}\cap D^-}\mathcal{F}_{g_{\beta},x_k}
 =\mathcal{F}_{g_{\alpha},x_k} \]
 which establishes property (4).

 \textit{Case 3.\;}
 $f_{\alpha}
= \frac{\varepsilon}{m_{2j+1}} E(\lambda_{f_{2}^{\alpha}}f_{1}^{\alpha}+
 f_{2}^{\alpha}+\ldots+
 \lambda_{f_{n_{2j+1}}^{\alpha}}f_{n_{2j+1}-1}^{\alpha}
 + f_{n_{2j+1}}^{\alpha})\in K_{\phi}$ where
$\{f_\beta:\;\beta\in S_{\alpha}\}=\{Ef^{\alpha}_i:
Ef^{\alpha}_{i}\not=0,\;1\leq i\leq n_{2j+1}\}$,
$\varepsilon\in\{-1,1\}$, $E$ is an interval and $\phi$ is a
special sequence of length $n_{2j+1}$.

Let $\phi=(z_{1},f_{1},\ldots, z_{n_{2j+1}},f_{n_{2j+1}})$.
Without loss of generality we assume that $E=\mathbb{N}$ and
$\varepsilon=1$. Let us observe that the definition of
$\{g_\beta:\beta\in D\}$ and the inductive assumptions yield
that for  $i\le n_{2j+1}/2$,
 \begin{enumerate}
 \item[(i)] $f_{2i-1}=f^{\alpha}_{2i-1}=g^{\alpha}_{2i-1}$.
 \item[(ii)] $w(f_{2i})=w(f^{\alpha}_{2i})=w(g^{\alpha}_{2i})$.
 \item[(iii)] $\supp(f_{2i})=\supp(f^{\alpha}_{2i})=\supp(g^{\alpha}_{2i})$.
 \end{enumerate}
 We define
$$
g_{\alpha}=\frac{1}{m_{2j+1}}\left(\lambda_{g^{\alpha}_2}f_1+
g^{\alpha}_2+ \lambda_{g^{\alpha}_4}f_3+g^{\alpha}_4+\cdots
+\lambda_{g^{\alpha}_{n_{2j+1}}}f_{n_{2j+1}-1}+
g^{\alpha}_{n_{2j+1}}\right)
$$
where $\{g_\beta:\;\beta\in S_{\alpha}\}=\{g^{\alpha}_i:\;1\le
i\le n_{2j+1}\}$ while  $\lambda_{g^{\alpha}_{2i}}$
are defined as follows:\\
(5)\: If $g^{\alpha}_{2i}(z_{2i})\neq 0$ then
$\lambda_{g^{\alpha}_{2i}}=
g^{\alpha}_{2i}(m_{\sigma(\phi_{2i-1})}z_{2i})$.\\
(6)\: If $g^{\alpha}_{2i}(z_{2i})= 0$ and $f^{\alpha}_{2i-1}=f_{\beta}$,
there are two cases
\begin{enumerate}
 \item[a)] If
$\beta\in \mathcal{F}_f$ or $\beta\not\in \mathcal{F}_f$ and
$\range(f_{\beta})\cap\range(x_k)=\emptyset$ for all $k$ we set
$\lambda_{g^{\alpha}_{2i}}=\frac{1}{n_{2j+1}^2}$.
 \item[b)] If $\beta\not\in \mathcal{F}_f$
and there exists (unique) $k$ such that
$\range(f_{\beta})\cap\range(x_k)\neq\emptyset$ then we set
$\lambda_{g^{\alpha}_{2i}}=\varepsilon_k
\lambda_{f^{\alpha}_{2i}}$.
 \end{enumerate}
Let us observe that in the case (6) b), as follows from the
comments after Step 1, $g_{\beta^+}=\varepsilon_k f_{\beta^+}$
hence $f_{\beta^+}(z_{2i})=0$ if and only if
$g_{\beta^+}(z_{2i})=0$.

 From the above definition of $\lambda_{g^{\alpha}_{2i}}$, $1\le
i\le n_{2j+1}/2$ and (i),(ii),(iii), we obtain that the functional
$g_{\alpha}$ belongs to $K_{\phi}\subset K$.

Properties (1) and (2) are obvious for $g_{\alpha}$ and we check
the rest. First we establish  property (4).

Let $k$ be given. From Remark~\ref{34} (1) it follows that there exists at most one depended
couple $f^{\alpha}_{2i-1}, f^{\alpha}_{2i}$ w.r.t. $f$ and $x_{k}$. Moreover if such a depended couple,
$f^{\alpha}_{2i-1}, f^{\alpha}_{2i}$, exists then for every
$i^{\prime}\neq i$ it holds that
$\mathcal{F}_{f^{\alpha}_{2i^{\prime}},x_k}=\emptyset$. Therefore in this case we have that
\begin{equation}\label{g1}
\mathcal{F}_{f_{\alpha},x_k}=\mathcal{F}_{f^{\alpha}_{2i},x_k}
\cup\{\beta\}
\end{equation}
where $f^{\alpha}_{2i-1}=f_\beta$. In the case that no such
depended couple exists, it follows that
$\mathcal{F}_{f_{2i}^{\alpha},x_{k}}\not=\emptyset$ for at most
one $i$. This is a consequence of the definitions and the fact that the
functionals $(f_{i}^{\alpha})_{i}$ are successive. If such an $i$
exists then
\begin{equation}\label{g2}
\mathcal{F}_{f_{\alpha},x_k}=\mathcal{F}_{f^{\alpha}_{2i},x_k}
\end{equation}
The last alternative is that $\mathcal{F}_{f_{\alpha},x_k}=\emptyset$. This
description of $\mathcal{F}_{f_{\alpha},x_k}$ and the inductive
assumptions easily yield property (4). Namely, either $\mathcal{F}_{g_{\alpha}, x_{k}}=\mathcal{F}_{g^{\alpha}_{2i}, x_{k}}\cup\{\beta\}$ if \eqref{g1} holds, 
$\mathcal{F}_{g_{\alpha}, x_{k}}=\mathcal{F}_{g^{\alpha}_{2i}, x_{k}}$ if \eqref{g2} holds, or
$\mathcal{F}_{g_{\alpha}, x_{k}}=\emptyset$.

 Finally we check property (3). Fix a number $k$ and $i\le
n_{2j+1}/2$. If $g^{\alpha}_{2i}=g_{\beta}$ and $\beta\in D^-$
the inductive assumption
 provides
 \begin{equation}\label{102}
g^{\alpha}_{2i}(\varepsilon_k(x_k-y_k))=f^{\alpha}_{2i}(x_k-y_k).
 \end{equation}
 If $\beta\in D$ and $\range (f^{\alpha}_{2i})\cap \range(x_k)\neq
\emptyset$ then $g^{\alpha}_{2i}=\varepsilon_k f^{\alpha}_{2i}$
which yields (\ref{102}). Also if $\range (f^{\alpha}_{2i})\cap
\range(x_k)= \emptyset$ equality (\ref{102})
 trivially holds.

 In the case $g^{\alpha}_{2i-1}=g_{\beta}$, $\beta\in S_{\alpha}$
 we distinguish two subcases.
 First assume that $\beta\in \mathcal{F}_{f}$. Then
 $\supp (g^{\alpha}_{2i-1})=\supp (f^{\alpha}_{2i-1})$ and
 $\supp (f^{\alpha}_{2i-1})\cap \supp(x_k-y_k)=\emptyset$ therefore
\[
g^{\alpha}_{2i-1}(\varepsilon_k (x_k-y_k))=0=
f^{\alpha}_{2i-1}(x_k-y_k).
\]
The second subcase  is $\beta\not\in \mathcal{F}_{f}$. As we have
explained in the comments after Step 1 that means that either
$\range (f_\beta)\cap \range(x_k)= \emptyset$, hence everything
trivially holds, or $\beta,{\beta}^{+}\in D$,
$g_{{\beta}^{+}}=\varepsilon_k f_{{\beta}^{+}}$ and
$\lambda_{g^{\alpha}_{2i}}=\varepsilon_k\lambda_{f^{\alpha}_{2i}}$.
From these observations we conclude that
\[
\lambda_{g^{\alpha}_{2i}}g^{\alpha}_{2i-1}(\varepsilon_k(x_k-y_k))=
     \lambda_{f^{\alpha}_{2i}}f^{\alpha}_{2i-1}(x_k-y_k).
\]
All these derive the desired equality, namely
 \[
g_{\alpha}(\varepsilon_k(x_k-y_k))=f_{\alpha}(x_k-y_k).
\]
The inductive construction and the entire proof of the lemma is
complete.
 \end{proof}
\begin{proof}[Proof of Proposition~\ref{31}]
Let $(\sigma_{\ell})_{\ell}$ be a decreasing sequence of positive
numbers such that $\sum_{\ell}\sigma_{\ell}\leq 1/8$. For each
$\ell\in\mathbb{N}$ we select $k_{\ell}$ such that
$\norm[x_{k_{\ell}}]_{\widetilde{K}}<\sigma_{\ell}$. For
simplicity we assume that the entire sequence $(x_{\ell})$
satisfies the above condition. Let
$\sum_{\ell=1}^{d}b_{\ell}x_{\ell}$ be a finite linear combination
which maximizes the norm of all vectors of the form
$\sum_{\ell=1}^{d}c_{\ell}x_{\ell}$ with $\vert
c_{\ell}\vert=\vert b_{\ell}\vert$. Assume furthermore that
$\norm[\sum_{\ell=1}^{d}b_{\ell}x_{\ell}]=1$ and let $f\in K$ with
$f(\sum_{\ell=1}^{d}b_{\ell}x_{\ell})\geq 3/4$.
  Choose $\{\varepsilon_{\ell}\}_{\ell=1}^{d}\in \{-1,1\}^{d}$
 and consider the vector
 $\sum_{\ell=1}^{d}\varepsilon_{\ell}b_{\ell}x_{\ell}$.
 Lemma~\ref{l1.6} yields that
 there exists $g\in K$ and that for each $\ell=1,\ldots,d$,
 there exists a vector $y_{\ell}$ such that
 \begin{equation}\label{e5}
g(\sum_{\ell=1}^{d}\varepsilon_{\ell}b_{\ell}(x_{\ell}-y_{\ell}))=
 f(\sum_{\ell=1}^{d}b_{\ell}(x_{\ell}-y_{\ell}))\,.
 \end{equation}
 Also Lemma~\ref{l1.5} and Lemma~\ref{l1.6}(2) and (3) yield that
 \begin{equation*}
 \vert g(y_{\ell})\vert\leq 2\sigma_{\ell}\quad \text{and}\quad
 \vert f(y_{\ell})\vert\leq 2\sigma_{\ell}\,\,\,\text{for
 all}\,\,\ell=1,\ldots,d.
 \end{equation*}
 Hence
 \begin{align*}
 \norm[\sum_{\ell=1}^{d}\varepsilon_{\ell}b_{\ell}x_{\ell}]& \geq \vert
  g(\sum_{\ell=1}^{d}\varepsilon_{\ell}b_{\ell}x_{\ell})\vert \geq \vert
g(\sum_{\ell=1}^{d}\varepsilon_{\ell}b_{\ell}(x_{\ell}-y_{\ell}))\vert-
 \sum_{\ell=1}^{d}\vert g(y_{\ell})\vert\\ &\geq \vert
 f(\sum_{\ell=1}^{d}b_{\ell}x_{\ell})\vert-
 \sum_{\ell=1}^{d}\vert g(y_{\ell})\vert- \sum_{\ell=1}^{d}\vert
 f(y_{\ell})\vert\geq 3/4-2/4=1/4\,.
 \end{align*}
 This completes the proof of the proposition.
\end{proof}
\section{The space $X_{ius}$ is indecomposable}
 In the last section we shall show that the space $X_{ius}$ is
 indecomposable. This will be a consequence of a stronger result
 concerning the structure of the space $\mathcal{B}(X_{ius})$ of the
 bounded linear operators acting on $X_{ius}$.
 The proof adapts techniques related to H.I. spaces as they were
 presented in \cite{AT}. Thus we will first consider the auxiliary space
 $X_u$ and we will estimate the norm of certain averages of its basis.
 Next we will use the basic inequality to reduce upper estimation on certain
 averages to the previous results. Finally
 we shall compute the norms of linear combinations related to special
 sequences.

\medskip

 {\bf The auxiliary spaces $X_u$, $X_{u,k}$}

 We begin with the definition of the
 space $X_{u}$ which will be used to provide us upper
 estimations for certain averages in the space $X_{ius}$.

The space $X_{u}$ is the mixed Tsirelson space
$T[(\mc{A}_{4n_{j}},\frac{1}{m_{j}})_{j=1}^{\infty}]$. The norming
set $W$ of $X_{u}$ is defined in a similar manner as the set $K$.
\par We set $W^{j}_{0}=\{\pm e_{n}^*: n\in\mathbb{N}\}\cup \{0\}$,
for $j\in\mathbb{N}$ , $W_{0}=\cup_{j}W^{j}_{0}$. In the general
inductive step we define
$$
W_{n}^{j}=W_{n-1}^{j}\cup \{\frac{1}{m_{j}}\sum_{i=1}^{d}f_{i}:
 d\leq 4n_{j}, f_{1}<\ldots<f_{d}\in W_{n-1}\}
$$
and $W_{n}=\cup_{j}W_{n}^{j}$. Finally let $W=\cup_{n}W_{n}$. The
space $X_{u}$ is the completion of $(c_{00},\norm[\cdot]_W)$ where
$$
\norm[x]_{W}=\sup\{\langle f,x\rangle:
f\in W\}\,.
$$
 It is clear that the norming set $K$ of
 the space $X_{ius}$ is a subset of the convex hull of $W$. Hence we
 have  that $\norm[x]_{K}\leq\norm[x]_{W}$
 for every $x\in c_{00}$.

 We also need the spaces
 $X_{u,k}=T[(\mc{A}_{4n_{j}},\frac{1}{m_{j}})_{j=1}^{k}]$. The norm
of such a space is denoted by $\norm[\cdot]_{u,k}$ and it is defined in
 a similar
 manner as the norm of $X_u$. Namely we define $W_n^j$, $n\in\mathbb{N}$,
 $1\le j\le k$ as above and $W_n^{(k)}=\bigcup\limits_{j=1}^kW_n^j$. The
 norming set is $W^{(k)}=\bigcup\limits_{n=0}^{\infty}W_n^{(k)}$.
 Spaces of this form have been studied in \cite{BD} and it has been shown
 that  such a space is either isomorphic to some $\ell_p$, $1<p<\infty$,
 or to $c_0$.

 Before stating the next lemma we introduce some notations.
 For each $k\in\mathbb{N}$ we set $q_k=\frac{1}{\log_{4n_k}m_k}$.
 and $p_k=\frac{1}{1-\log_{4n_k}m_k}$ its conjugate.

 \begin{lemma}\label{103}
For the sequences $(m_j)_j$, $(n_j)_j$ used in the definition of
$X_{ius}$ and $X_u$, $X_{u,k}$  the following hold:
 \begin{enumerate}
\item[(1)\;] The sequence $(q_j)_j$  strictly increases to infinity.
\item[(2)\;] For $x=\sum a_{\ell} e_{\ell}\in c_{00}$, $\norm[x]_{u,k}\le
\norm[x]_{p_k}$.
\item[(3)\;] $\norm[\frac{1}{n_{k+1}}\sum\limits_{i=1}^{n_{k+1}}e_i]_{p_k}
\leq\frac{1}{m_{k+1}^3}$.
 \end{enumerate}
 \end{lemma}
 \begin{proof} (1)\,\, Using that $m_{j+1}=m_j^5$ and $n_{j+1}=(4n_j)^{s_j}$
 and the fact that $s_j$ increases to infinity we have that
 \[ q_{j+1}=\frac{1}{\log_{4n_{j+1}}m_{j+1}} =
            \frac{1}{\log_{4(4n_j)^{s_j}}m_j^5}>
\frac{1}{\frac{5}{s_j}\log_{4n_j}m_j}=\frac{s_j}{5}q_j  \]
 hence $(q_j)_j$ strictly increases to infinity.

 (2)\,\, We inductively show that  for $f\in W_n^{(k)}$
  \[    |f(\sum a_{\ell} e_{\ell})|\le\norm[\sum a_{\ell} e_{\ell}]_{p_k}.\]
  For $n=0$ it is trivial. The general inductive step goes as follows:
  for $f\in W_{n+1}^{(k)}$
$$
f(\sum a_{\ell} e_{\ell})=\frac{1}{m_j}\sum\limits_{i=1}^df_i
(\sum a_{\ell} e_{\ell})
$$
where $f_1<f_2<\cdots< f_d$, $d\le 4n_j$ for some $j\le k$. We set
$E_i=\range(f_i)$ and from our inductive assumption
  and H\"{o}lder inequality we obtain that
  \begin{equation*}
 |f(\sum a_{\ell} e_{\ell})|\le
 \frac{1}{m_j}\sum\limits_{i=1}^d
 \norm[\sum\limits_{\ell\in E_i}a_{\ell}e_{\ell}]_{p_k}
\le\frac{d^{\frac{1}{q_j}}}{m_j} \big(\sum\limits_{i=1}^d
\norm[\sum\limits_{\ell\in E_i}a_{\ell}e_{\ell}
]_{p_k}^{p_j}\big)^{\frac{1}{p_j}}\,.
 \end{equation*}
 Using that $p_k\le p_j$ and $m_j=(4n_j)^{\frac{1}{q_j}}$ we obtain
 inequality (2).

 (3)
\[\norm[\frac{1}{n_{k+1}}\sum\limits_{i=1}^{n_{k+1}}e_i]_{p_k}
 \le \frac{1}{n_{k+1}^\frac{1}{q_k}}=\frac{1}{(4n_k)^{\frac{s_k}{q_k}}}
 =\frac{1}{m_k^{s_k}}\le \frac{1}{m_{k+1}^3}  .\]
  (Recall that $2^{s_k}\ge m_{k+1}^3$).
 \end{proof}

The tree $\mathcal{T}_{f}$ of $f\in W$ is defined in a similar
manner as for $f\in K$.
\begin{lemma} \label{42}
Let $f\in W$ and $j\in\mathbb{N}$. Then
\begin{equation}\label{basise1}
\vert f(\frac{1}{n_{j}}\sum_{i=1}^{n_{j}}e_{k_{i}})\vert
\leq
\begin{cases}
\frac{2}{w(f)\cdot m_{j}},\quad &\text{if}\,\,\,w(f)<m_{j}\\
\frac{1}{w(f)},\quad &\text{if}\,\,\,w(f)\geq m_{j}\,.
\end{cases}
\end{equation}
If moreover we  assume that there exists a tree
$(f_{\alpha})_{\alpha\in\mc{A}}$ of $f$, such that $w(f_{\alpha})\not= m_{j}$
for every $\alpha\in\mc{A}$,  we have that
\begin{equation}\label{basise2}
 \vert
f(\frac{1}{n_{j}}\sum_{i=1}^{n_{j}}e_{k_{i}})\vert
\leq\frac{2}{m_{j}^{3}}\,.
\end{equation}
In particular the above upper estimations holds for every $f\in
K$.
\end{lemma}
\begin{proof}
 If $w(f)\geq m_{j}$ the estimation is an immediate consequence
 of the fact that
 $\norm[f]_{\infty}\leq 1/w(f)$. Let $w(f)<m_{j}$ and
 $(f_{\alpha})_{\alpha\in\mc{A}}$ be a tree of $f$. We
 set $$ B=\{i:\text{there exists}\,\,\alpha\in\mc{A}\,\,\text{with}\,\,
k_{i}\in\text{supp}f_{\alpha}\,\,\text{and}\,\,w(f_{\alpha})\geq m_{j}\} $$
 Then we have that
 \begin{equation}\label{eb1}
 \vert f(\frac{1}{n_{j}}\sum_{i\in B}e_{k_{i}})\vert\leq
 \frac{1}{w(f)m_{j}}\,\,\,.
 \end{equation}
 To estimate $|f(\frac{1}{n_{j}}\sum_{i\in B^{c}}e_{k_{i}})|$,
  we observe that $f|_{\{k_i:\;i\in B^c\}}\in W^{(j-1)}$
  (the norming set of $X_{u,j-1}$)
   hence   Lemma \ref{103} yields that
  \begin{equation}\label{eb2}
 \vert f(\frac{1}{n_{j}}\sum_{i\in B^{c}}e_{k_{i}})\vert
 \leq \frac{1}{m_{j}^{3}}\,.
 \end{equation}

 Combining \eqref{eb1} and \eqref{eb2} we obtain \eqref{basise1}.

 To see   (\ref{basise2}) we define the set
 \[ B=\{i:\text{there exists}\,\,\alpha\in\mc{A}\,\,\text{with}\,\,
k_{i}\in\text{supp}f_{\alpha}\,\,\text{and}\,\,w(f_{\alpha})\geq m_{j+1}\} \]
 and we conclude that
 \begin{equation}\label{104}
 | f(\frac{1}{n_j}\sum\limits_{i\in B}e_{k_i})|\le \frac{1}{m_{j+1}}
 <\frac{1}{m_j^3}\,\,.
 \end{equation}
 Furthermore from our assumption $w(f_{\alpha})\neq m_j$ for every $\alpha\in\mathcal{A}$
 we conclude that $f|_{\{k_i:\;i\in B^c\}}\in W^{(j-1)}$.
  This yields that
 the corresponding of (\ref{eb2}) remains valid and combining
 (\ref{eb2}) and (\ref{104}) we obtain (\ref{basise2}).
 \end{proof}

{\bf The basic inequality and its consequences}

Next we state and prove the basic inequality
 which is an adaptation of the
corresponding result from \cite{AT}. Actually the proof of the
present statement is easier than the original one, due mainly to
the low complexity of the family $\mathcal{A}_n$ (in \cite{AT} are
studied spaces defined with  use of the Schreier families
$(\mathcal{S}_{\xi})_{\xi<\omega_1}$) and also since the
definition of the norming set $K$ does not involve convex
combinations. The role of this result is important since it
includes most of the necessary computations (unconditional or
conditional).

Recall that $K$ and $W$ denote the norming sets of
$X_{ius}$ and $X_u$ respectively.

\begin{proposition}\label{bin}(Basic inequality)
Let $(x_{k})$ be a block sequence in $X_{ius}$, $(j_{k})$ be a strictly
increasing sequence of positive integers, $(b_{k})\in c_{00}$,
$C\geq 1$ and $\varepsilon>0$ such that

$a)$ $\norm[x_{k}]\leq C$ for every $k$.

$b)$ For every $k\geq
1$,\,\,$\#(\text{supp}x_{k})\frac{1}{m_{j_{k+1}}} \leq\varepsilon$.

$c)$ For every $k\geq 1$, for all $f\in K$ with $w(f)<m_{j_{k}}$,
we have that $\vert f(x_{k})\vert\leq\frac{C}{w(f)}$.

 Then for every $f\in K$ there exists $g_{1}$ such that
 $g_{1}=h_{1}$ or $g_{1}=e^{*}_{t}+h_{1}$
 where $t\not\in\text{supp}h_{1}$, $h_{1}\in W$, $w(h_1)=w(f)$,
 and $g_{2}\in c_{00}$ with $\Vert
 g_{2}\Vert_{\infty}\leq \varepsilon$ such that
 \begin{equation}
 \vert f(\sum b_{k}x_{k})\vert \leq C(g_{1}+g_{2})(\sum\vert
 b_{k}\vert e_{k})\,\,,
 \end{equation}
 and $\text{supp}g_{1}$, $\text{supp}g_{2}$ are contained in
 $\{k:\text{supp}(f)\cap\text{range}(x_{k})\not=\emptyset\}$.

 $d)$ If we additionally  assume that for some
 $j_{0}\in\mathbb{N}$ we have that
 \begin{equation}
 \vert f(\sum\limits_{k\in E} b_{k}x_{k})\vert \leq C(\max_{k\in E}\vert
 b_{k}\vert+\varepsilon\sum_{k\in E}\vert b_{k}\vert)\,\,,
 \end{equation}
 for every interval $E$ of positive integers
 and every $f\in K$ with $w(f)=m_{j_{0}}$, then $h_{1}$ may be
  selected  to have a tree $(h_{\alpha})_{\alpha\in\mc{A}_{1}}$
   such that $w(h_{\alpha})\not=m_{j_{0}}$ for every $\alpha\in\mc{A}_{1}$.
 \end{proposition}
 Our intention is to apply the above inequality in order to obtain
upper estimations for $\ell_{1}-$averages of rapidly increasing sequences.
Observe that the above proposition reduces this problem
to the estimations of the functionals $g_1,g_2$ on a corresponding average of the basis in the space $X_{u}$.

 The proof in the general case, assuming only $a),b), c)$,
and in the special case,
where additionally $d)$ is assumed, is the same. We will make the proof only in the special case. The proof in the general case arises by omitting any reference to the question whether a functional has weight $m_{j_0}$ or not.  For the rest of the proof we assume that there exists $j_0\in\mathbb{N}$ such that condition d) in the
 statement of Proposition is fulfilled.
\begin{proof}[Proof of Proposition~\ref{bin}]
Let $f\in K$ and let $\mc{T}_{f}=(f_{\alpha})_{\alpha\in\mc{A}}$ be a
tree of $f$. For every $k$ such that
$\text{supp}(f)\cap\text{range}(x_{k})\not=\emptyset$ we define the set $A_{k}$ as follows:
 \begin{align*}
A_{k}=\Big\{\alpha\in\mc{A}:&\,\,(i)\,\, \text{supp}f_{\alpha}
\cap\text{range}(x_{k})=
\text{supp}(f)\cap\text{range}(x_{k}),
\\
& (ii)\,\, \text{ for all}\,
\gamma\prec\alpha,\,\,w(f_{\gamma})\not=m_{j_{0}}\,,
\\
& (iii)\,\, \text{there is no}\,\,\beta\in S_{\alpha}\,\,
\text{such that}
\\
&\qquad \text{supp}(f_{\alpha})\cap\text{range}(x_{k})=
\text{supp}(f_{\beta})\cap\text{range}(x_{k})\,\,
\text{if}\,\, w(f_{\alpha})\not=m_{j_{0}} \Big\}\,.
\end{align*}
 From the definition, it follows easily that for every $k$ such that $\text{supp}(f)\cap\text{range}(x_{k})\not=\emptyset$
$A_{k}$ is a singleton.

We recursively define sets $(D_{\alpha})_{\alpha\in\mc{A}}$ as
follows.

For every terminal node $\alpha$ of the tree we set
$D_{\alpha}=\{k:\alpha\in A_{k}\}$.
For every non terminal node $\alpha$ we define,
\begin{equation*}
 D_{\alpha}=\{k:\alpha\in A_{k}\}\cup  \cup_{\beta\in
 S_{\alpha}}D_{\beta}\,\,.
 \end{equation*}
The following are easy
consequences of the definition.
\begin{enumerate}
\item[i)]
 If $\beta\prec\alpha$, $D_{\alpha}\subset D_{\beta}$.
\item[ii)] If $w(f_{\alpha})=m_{j_{0}}$, then $D_{\beta}=\emptyset$
 for all $\beta\succ\alpha$.
\item[iii)]
If $w(f_{\alpha})\not=m_{j_{0}}$, then for every $
\{ \{k\}:k\in D_{\alpha}\setminus \cup_{\beta\in
S_{\alpha}}D_{\beta}\}\cup\{D_{\beta}:
\beta\in S_{\alpha}\}$
is a family of successive subsets of $\mathbb{N}$.
\item[iv)] If $w(f_{\alpha})\not=m_{j_{0}}$, for every $k\in D_{\alpha}\setminus\cup_{\beta\in S_{\alpha}}D_{\beta}$
there exists $\beta\in S_{\alpha}$ such that
$\min\text{supp}x_{k}<
\min\text{supp}f_{\beta}\leq\max\text{supp}x_{k} $
and for $k^{\prime}\in D_{\alpha}\setminus \cup_{\beta\in S_{\alpha}}D_{\beta}$
different form $k$
the corresponding
 $\beta^{\prime}$ is different from $\beta$.
\end{enumerate}
Inductively for every $\alpha\in\mc{A}$ we define $g_{\alpha}^{1}$ and $g_{\alpha}^{2}$ such that
\renewcommand{\labelenumi}{\arabic{enumi}}
\begin{enumerate}
\item[(1)] For every $\alpha\in\mc{A}$,\,\,
$\text{supp}g_{\alpha}^{1}$ and $\text{supp}g_{\alpha}^{2}\subset
D_{\alpha}$.
\item[(2)] If $w(f_{\alpha})=m_{j_{0}}$,
$g_{\alpha}^{1}=e^{*}_{k_{\alpha}}$, where $\vert
b_{k_{\alpha}}\vert= \max_{k\in D_{\alpha}}\vert b_{k}\vert$ and
$g_{\alpha}^{2}=\varepsilon\sum_{k\in D_{\alpha}}e^{*}_{k}$\,.
\item[(3)] If $w(f_{\alpha})\not=m_{j_{0}}$,
$g_{\alpha}^{1}=h_{\alpha}$ or
$g_{\alpha}^{1}=e^{*}_{k_{\alpha}}+h_{\alpha}$ where
$k_{\alpha}\not\in\text{supp}h_{\alpha}$, $h_{\alpha}\in W$ and
$w(h_{\alpha})=w(f_{\alpha})$.
\item[(4)] For every $\alpha\in\mc{A}$
the following inequality holds
$$
\vert f_{\alpha}(\sum_{k\in D_{\alpha}}b_{k}x_{k})\vert\leq C(g_{\alpha}^{1}+g_{\alpha}^{2})(\sum_{k\in D_{\alpha}}\vert b_{k}\vert e_{k})\,\,.
$$
\end{enumerate}
For every terminal node we set $g_{\alpha}^{1}=g_{\alpha}^{2}=0$
if $D_{\alpha}=\emptyset$, otherwise we set $g_{\alpha}=e^{*}_{k}$
if $D_{\alpha}=\{k\}$ and $g_{\alpha}^{2}=0$. Assume that we have
defined the functionals $g_{\beta}^{1}$ and $g_{\beta}^{2}$,
satisfying $(1)-(4)$, for every $\beta\in\mc{A}$ with
$\vert\beta\vert=k$, and let $\alpha\in\mc{A}$ with
$\vert\alpha\vert=k-1$. If $D_{\alpha}=\emptyset$ we set
$g_{\alpha}^{1}=g_{\alpha}^{2}=0$. Let $D_{\alpha}\not=\emptyset$.
We distinguish two cases.

\textit{Case 1}. $w(f_{\alpha})=m_{j}\not= m_{j_{0}}$.

 Let $T_{\alpha}= D_{\alpha}\setminus\cup_{\beta\in S_{\alpha}}
 D_{\beta}=\{k: \alpha\in A_{k}\}$. We
 set $T_{\alpha}^{2}=\{k\in T_{\alpha} : m_{j_{k+1}}\leq m_{j}\}$ and
 $T_{\alpha}^{1}=T_{\alpha}\setminus T_{\alpha}^{2}$. In the
 pointwise estimations we
 shall make below, we shall discard the coefficient $\lambda_{f_{2i}}$,
 which appears in the  definition of the special functionals,
 since $\vert\lambda_{f_{2i}}\vert\leq 1$.

 From condition $b)$ in the statement,
 it follows that for each $k\in T_{\alpha}^{2}$
 \begin{equation}\label{ee1}
\vert f_{\alpha}(x_{k})\vert \leq
\#(\text{supp}x_{k})\norm[f_{\alpha}]_{\infty}\leq
\#(\text{supp}x_{k})\frac{1}{m_{j}}\leq\varepsilon\leq C\varepsilon\,\,.
\end{equation}
We define
$$
g_{\alpha}^{2}= \varepsilon\sum_{k\in
T_{\alpha}^{2}}e_{k}^{*}+\sum_{\beta\in S_{\alpha}}g_{\beta}^{2}\,\,.
$$
 We observe that $\norm[g_{\alpha}^{2}]_{\infty}\leq\varepsilon$, and
that $\vert f_{\alpha}(x_{k})\vert\leq C\varepsilon=Cg_{\alpha}^{2}(e_{k})$, for
every $k\in T_{\alpha}^{2}$.

 Let $T_{\alpha}^{1}=\{k_{1}<k_{2}<\ldots<k_{l}\}$. By the definition of
 $T_{\alpha}^{1}$   we have
 that $m_{j}<m_{j_{k_{2}}}<m_{j_{k_{3}}}<\ldots<m_{j_{k_{l}}}$.
 Thus  condition $c)$ in the  statement implies that
\begin{equation}\label{ee2}
\vert f_{\alpha}(x_{k_{i}})\vert\leq \frac{C}{m_{j}}= \frac{1}{m_{j}}e^{*}_{k_{i}}(Ce_{k_{i}}),\quad\text{for every}\,\, 2\leq i\leq l\,.
\end{equation}
We set
$$
g_{\alpha}^{1}=e_{k_{1}}^{*}+ \frac{1}{m_{j}}(\sum_{i=2}^{l}e^{*}_{k_{i}}+\sum_{\beta\in
S_{\alpha}}g_{\beta}^{1})\,.
$$
(The term $e^{*}_{k_{1}}$ does not appear if
$w(f_{\alpha})<m_{j_{k}}$ for every $k\in T_{\alpha}$). We have to
show that
$h_{\alpha}= \frac{1}{m_{j}}(\sum_{i=2}^{l}e^{*}_{k_{i}}+\sum_{\beta\in
S_{\alpha}}g_{\beta}^{1})\in W$. From the inductive hypothesis, we
have that $g^{1}_{\beta}=h_{\beta}$ or
$g^{1}_{\beta}=e^{*}_{k_{\beta}}+h_{\beta}$, $h_{\beta}\in W$, for
every $\beta\in S_{\alpha}$. For $\beta\in S_{\alpha}$, such that
$g_{\beta}^{1}=e^{*}_{k_{\beta}}+h_{\beta}$, let
$E_{\beta}^{1}=\{n\in\mathbb{N}: n<k_{\beta}\}$ and
$E_{\beta}^{2}=\{n\in\mathbb{N}: n>k_{\beta}\}$. We set
$h_{\beta}^{1}=E_{\beta}^{1}h_{\beta}$,
$h_{\beta}^{2}=E_{\beta}^{2}h_{\beta}$.
For every $\beta$ such that $g_{\beta}^{1}=e^{*}_{k_{\beta}}+h_{\beta}$, the functionals $h_{\beta}^{1}$, $e^{*}_{k_{\beta}}$,
$h_{\beta}^{2}$ are successive  belonging to $W$, and for
$\beta\not=\beta^{\prime}\in S_{\alpha}$ the corresponding
functionals have disjoint range, since $\text{supp}g_{\beta}^{1}$ is
an interval, remark (iii) after the definition of $D_{\alpha}$.
From the remark iv) after the definition of $D_{\alpha}$ we have
that $\#T_{\alpha}^{1}\leq n_{j}$. It follows that $$ \#(\{
e^{*}_{k_{i}}, 2\leq i\leq l\}\cup\{ e^{*}_{k_{\beta}},
h^{1}_{\beta}, h^{2}_{\beta}:\, \beta\in S_{\alpha},
g_{\beta}=e^{*}_{k_{\beta}}+h_{\beta}\} \cup\{h_{\beta} :\,
\beta\in S_{\alpha}, g_{\beta}=h_{\beta}\} )\leq 4n_{j}\,. $$
Therefore
$h_{\alpha}=\frac{1}{m_{j}}(\sum_{i=2}^{l}e^{*}_{k_{l}}+\sum_{\beta\in
S_{\alpha}}g_{\beta}^{1})\in W$. It remains to show property $4).$
By \eqref{ee2} we have that $\vert f_{\alpha}(x_{k_{i}})\vert\leq
Cg_{\alpha}^{1}(e_{k_i})$ for every $2\leq i\leq l$, while $$
\vert f_{\alpha}(x_{k_{1}})\vert\leq \norm[x_{k_{1}}]\leq
Ce^{*}_{k_{1}}(e_{k_{1}})=g_{\alpha}^{1}(Ce_{k_{1}})\,. $$ We also
have that
\begin{align*}
\vert f_{\alpha}(\sum_{k\in \cup_{\beta\in
S_{\alpha}}D_{\beta}}b_{k}x_{k})\vert & \leq \sum_{\beta\in
S_{\alpha}}\vert f_{\alpha} (\sum_{k\in D_{\beta}}b_{k}x_{k})\vert \\
&\leq \frac{1}{m_{j}}\sum_{\beta\in S_{\alpha}}\vert
f_{\beta}(\sum_{k\in D_{\beta}}b_{k}x_{k})\vert\\ & \leq
\frac{1}{m_{j}}\sum_{\beta\in
S_{\alpha}}(g_{\beta}^{1}+g_{\beta}^{2})(C\sum_{k\in
D_{\beta}}\vert b_{k}\vert e_{k})\vert\\ & \leq (g_{\alpha}^{1}+
g_{\alpha}^{2})(C\sum_{k\in D_{\alpha}}\vert b_{k}\vert
e_{k})\vert\,.
\end{align*}
\textit{Case 2}. $w(f_{\alpha})=m_{j_{0}}$. \noindent In this case we have that $D_{\alpha}$ is an
 interval of the positive integers and $D_{\gamma}=\emptyset$, for every $\gamma\succ\alpha$. Let
 $k_{\alpha}$ such that $b_{k_{\alpha}}=\max_{k\in D_{\alpha}}\vert b_{k}\vert$. We set
 $$
 g_{\alpha}^{1}=e^{*}_{k_{\alpha}}\,\,\,\text{and}\,\,\, g_{\alpha}^{2}=\varepsilon\sum_{k\in
 D_{\alpha}}e^{*}_{k}\,.
 $$
 Then we have that
$$ \vert f_{\alpha} (\sum_{k\in D_{\alpha}}b_{k}x_{k})\vert \leq
C(\max_{k\in D_{\alpha}}\vert b_{k}\vert+ \varepsilon\sum_{k\in
D_{\alpha}}\vert
b_{k}\vert)=(g_{\alpha}^{1}+g_{\alpha}^{2})(C\sum_{k\in
D_{\alpha}}\vert b_{k}\vert e_{k})\,\,.
$$
\end{proof}
\begin{definition}
 Let $k\in\mathbb{N}$. A vector $x\in c_{00}$ is said to be
 a $C-\ell_{1}^{k}$ average if there
 exists $x_{1}<\ldots<x_{k}$, $\norm[x_{i}]\leq C\norm[x]$ and
 $x=\frac{1}{k}\sum_{i=1}^{k}x_{i}$. Moreover, if $\norm[x]=1$ then
 $x$ is called a normalized $C-\ell_{1}^{k}$ average.
\end{definition}
\begin{lemma}\label{106}
Let $j\geq 1$, $x$ be an $C-\ell_{1}^{n_{j}}$-average.  Then for
every $n\leq n_{j-1}$ and
 every
 $E_{1}<\ldots<E_{n}$,
we have that $$ \sum_{i=1}^{n}\norm[E_{i}x]\leq
C(1+\frac{2n}{n_{j}})
 <\frac{3}{2}C. $$
\end{lemma}
We refer to \cite{Sh}, (or \cite{GM}, Lemma 4), for a proof.
\begin{proposition}\label{p18}
For every  normalized block sequence $(y_{\ell})_{\ell}$ and every
$k\ge m_2$ there exists  a linear combination of
$(y_{\ell})_{\ell}$ which is a normalized $2-\ell_1^k$ average.
\end{proposition}
\begin{proof}
Given $k\ge m_2$ there exists $j\in\mathbb{N}$ such that
$m_{2j-1}<k\le m_{2j+1}$. Recall that $n_{2j+2}=(4n_{2j+1})^{s_{2j+1}}$
and $m_{2j+2}^3<2^{s_{2j+1}}$.
Hence setting $s=s_{2j+1}$ we have that $k^s\le n_{2j+2}$ and
 $2^{-s}< \frac{1}{m_{2j+2}}$.
  Observe that
 \begin{equation}\label{107}
\norm[\sum\limits_{i=1}^{k^s}y_i]\ge \frac{k^s}{m_{2j+2}}\,\,.
 \end{equation}
 Assuming that there is no normalized $2-\ell_1^k$ average in
 $\langle y_i:\; i\le k^s\rangle$ and following the proof of Lemma 3 in \cite{GM}
 we obtain that
 \begin{equation}\label{108}
\norm[\sum\limits_{i=1}^{k^s}y_i]<k^s\cdot 2^{-s}.
 \end{equation}
Since $2^{-s}< \frac{1}{m_{2j+2}}$, (\ref{107}) and (\ref{108})
derive  a contradiction.
 \end{proof}
\begin{definition}
A block sequence $(x_{k})$ in $X_{ius}$ is said to be a $(C,\varepsilon)$-rapidly
increasing sequence  (R.I.S.), if there, exists a strictly
increasing sequence $(j_{k})$ of positive integers such that

\noindent a) $\norm[x_{k}]\leq C$.

\noindent b) $\#(\range(x_{k}))\frac{1}{m_{j_{k+1}}}<\varepsilon$.

\noindent c) For every $k=1,2,\ldots$ and every $f\in K$ with
$w(f)<m_{j_{k}}$  we have that $\vert f(x_{k})\vert \le
\frac{C}{w(f)}$.
\end{definition}

 \begin{remark}
Let $(x_k)_k$ be  a block sequence in $X_{ius}$ such that each
$x_k$ is a normalized $\frac{2C}{3}-\ell_1^{n_{j_k}}$ average and
let $\varepsilon>0$ be such that for each $k$,
$\#(\range(x_{k}))\frac{1}{m_{j_{k+1}}}<\varepsilon$. Then Lemma \ref{106}
yields that condition (c) in the above definition is also
satisfied hence $(x_k)_k$ is a $(C,\varepsilon)$ R.I.S. In this case we
shall call $(x_k)_k$ as a $\mathbf{(C,\varepsilon)}$ {\bf R.I.S. of}
$\mathbf{\ell_1}$ {\bf averages}. Let also observe that
Proposition \ref{p18} ensures that for every block sequence
$(y_{\ell})_{\ell}$ and every $\varepsilon>0$ there exists $(x_k)_k$ which is  a $(3,\varepsilon)$ R.I.S. of
$\ell_1$ averages.
\end{remark}
\begin{proposition}\label{ris}
Let $(x_{k})_{i=1}^{n_{j}}$ be a $(C,\varepsilon)$- R.I.S such that
$\varepsilon\leq\frac{1}{n_{j}}$. Then

1) For every $f\in K$
$$
\vert f(\frac{1}{n_{j}}\sum_{k=1}^{n_{j}}x_{k})\vert
\leq
\begin{cases}
\frac{3C}{m_{j}w(f)}\,,\quad &\text{if}\,\,\, w(f)<m_{j}\\
\frac{C}{w(f)}+\frac{2C}{n_{j}}\,,&\text{if}\,\,\, w(f)\geq m_{j}\,.
\end{cases}
$$ In particular $\norm[\frac{1}{n_{j}}\sum\limits_{k=1}^{n_{j}}x_{k}]
\leq\frac{2C}{m_{j}}$.

2) If for  $j_{0}=j$ the assumption d) of the basic inequality
is fulfilled (Proposition~\ref{bin}), for a linear combination
$\frac{1}{n_{j}}\sum_{i=1}^{n_{j}}b_{i}x_{i}$, where $\vert
b_{i}\vert\leq 1$, then 
$$
\norm[\frac{1}{n_{j}}\sum_{i=1}^{n_{j}}b_{i}x_{i}]
\leq\frac{4C}{m_{j}^{3}}\,\,.
$$
3) If $(x_{i})_{i=1}^{n_{2j}}$ is a $(3,\varepsilon)$ rapidly increasing
sequence of $\ell_{1}$ averages then
\begin{equation} \label{105}
\frac{1}{m_{2j}}\leq \norm[\frac{1}{n_{2j}}\sum_{i=1}^{n_{2j}}x_{i}]
\leq\frac{6}{m_{2j}}\,\,.
\end{equation}
\end{proposition}
\begin{proof} The proof of 1) is an application of the basic
inequality and Lemma~\ref{42}. Indeed for $f\in K$, the basic
inequality yields that  there exist $h_{1}\in W$ with
$w(f)=w(h_{1})$, $t\in\mathbb{N}$ with $t\not\in\text{supp}h_{1}$,
and $h_{2}\in c_{00}$ with $\norm[h_{2}]_{\infty}\leq \varepsilon$, such
that
\begin{equation}\label{eris1}
 \vert f(\frac{1}{n_{j}}\sum_{k=1}^{n_{j}}x_{k})\vert \leq
 (e_{t}^{*}+h_{1}+h_{2})
 C(\frac{1}{n_{j}}\sum_{k=1}^{n_{j}}e_{k})\,.
\end{equation}
Using Lemma~\ref{42} and the fact that $\varepsilon\leq\frac{1}{n_{j}}$ we
obtain
\begin{equation}\label{eris2}
\vert f(\frac{1}{n_{j}}\sum_{k=1}^{n_{j}}x_{k})\vert \leq
\begin{cases}
 \frac{C}{n_{j}}+
\frac{2C}{w(f)m_{j}}+C\varepsilon\leq\frac{3C}{w(f)m_{j}}\,\,\,&
 \text{if}\,\,w(f)<m_{j}\\ \frac{C}{n_{j}}+
 \frac{C}{w(f)}+C\varepsilon\leq\frac{C}{w(f)}+ \frac{2C}{n_{j}}\,\,\,&
 \text{if}\,\,w(f)\geq m_{j}\,\,.
\end{cases}
\end{equation}
To prove 2) we observe that the basic inequality yields the
existence of $h_{1}$, $h_{2}$ such that $h_{1}$ has a tree
$(h_{\alpha})_{\alpha\in\mc{A}}$ such that
$w(h_{\alpha})\not=m_{j}$ for every $\alpha\in\mc{A}$ and
$\norm[h_{2}]_{\infty}\leq \varepsilon.$ This and Lemma~\ref{42} yield that
\begin{equation}\label{eris3}
\vert f(\frac{1}{n_{j}}\sum_{k=1}^{n_{j}}b_{k}x_{k})\vert \leq
(e_{t}^{*}+h_{1}+h_{2}) C(\frac{1}{n_{j}}\sum_{k=1}^{n_{j}}e_{k})
\leq
\frac{C}{n_{j}}+\frac{2C}{m_{j}^{3}}+C\varepsilon \leq\frac{4C}{m_{j}^{3}}\,\,.
\end{equation}
The upper estimation in 3) follows from 1) for $C=3$. For the
lower estimation in 3), for every $i\leq n_{2j}$ we choose a
functional $f_{i}$ belonging to the pointwise closure of $K$
such that $f_{i}(x_{i})=1$ and
$\text{range}(f_{i})\subset\text{range}(x_{i})$. Then it is easy to see
that the
functional $f=\frac{1}{m_{2j}}\sum\limits_{i=1}^{n_{2j}}f_{i}$
belongs to the same set and
provides the required  result.
\end{proof}

\begin{proposition}
The space $X_{ius}$ is reflexive.
\end{proposition}
\begin{proof}
As we have already explained after the definition of the norming set $K,$ the basis is boundedly complete.
Therefore to show that the space $X_{ius}$ is reflexive we need to prove that the basis is shrinking.

Assume on the contrary. Namely there exists
$x^*=w^*-\sum\limits_{n=1}^{\infty}b_ne_n^*$ and $x^*\not\in
\overline{ <e_n^*>}$. Then there exists $\varepsilon>0$ and successive
intervals $(E_k)_k$ such that $\norm[E_k x^{*}]>\varepsilon$. Choose $(x_k)_k$
in $X_{ius}$ such that $\supp(x_k)\subset E_k$, $\norm[x_k]=1$ and
$x^*(x_k)>\varepsilon$. It follows that every convex combination $\sum a_k
x_k$ satisfies
\begin{equation}\label{110}
\norm[\sum  a_k x_k]>\varepsilon .
\end{equation}
Next for $j$ sufficiently large
such that $\frac{4}{\varepsilon m_{2j}}<\varepsilon$ we
define $y_1,y_2,\ldots,y_{n_{2j}}$ a
$(\frac{2}{\varepsilon},\frac{1}{n_{2j}})$ R.I.S. of $\ell_1$ averages
 and each $y_i$ is some average of $(x_k)_k$.
Proposition \ref{ris} (1) yields that
\begin{equation}\label{109}
\norm[\frac{1}{n_{2j}}(y_1+y_2+\cdots+y_{n_{2j}})]\leq
\frac{4}{m_{2j}\varepsilon}<\varepsilon.
\end{equation}
Clearly (\ref{109}) contradicts (\ref{110}) and the basis is
shrinking.
 \end{proof}

{\bf The structure of $\mc{B}(X_{ius})$}

 \begin{definition}\label{depseq}
A sequence $\chi=(x_1,f_1,x_2,f_2,\ldots,x_{n_{2j+1}},f_{n_{2j+1}})$
is said to be a {\bf depended sequence of length} $\mathbf{n_{2j+1}}$
if the following conditions are fulfilled
\begin{enumerate}
\item[(i)] There exists $\phi=(x_1,f_1,y_2,f_2,\ldots,x_{2i-1},f_{2i-1},
y_{2i},f_{2i},\ldots ,y_{n_{2j+1}},f_{n_{2j+1}})$ special sequence of length $n_{2j+1}$ such that
$\supp(y_{2i})=\supp(x_{2i})$ and
$\norm[y_{2i}-x_{2i}]\leq \frac{1}{n_{j_{2i}}^2}$
where for $1\leq i< n_{2j+1}$, $j_{i+1}=\sigma(\phi_i)$.
\item[(ii)] For $i\leq n_{2j+1}/2$ we have that
$$
x_{2i}= \frac{c_{2i}}{n_{j_{2i}}}\sum\limits_{l=1}^{n_{j_{2i}}}x^{2i}_l
$$
where $(x^{2i}_l)_l$ is a $(3,\frac{1}{n_{j_{2i}}})$ R.I.S. of $\ell^1$ averages, $c_{2i}\in(0,1)$.
\item[(iii)] $f_{2i}(x_{2i})\geq \frac{1}{12 m_{j_{2i}}}$.
\end{enumerate}
\end{definition}
The following is a consequence of the previous results, and we sketch the proof of it.
\begin{lemma}\label{412}
Let $(y_k)_k$ be a normalized block sequence in $X_{ius}$ and
$(e_n)_{n\in M}$ be a subsequence of its basis. Then for all $j\in\mathbb{N}$
there exists a depended sequence
$$
\chi=(x_1,f_1,x_2,f_2,\ldots,x_{n_{2j+1}},f_{n_{2j+1}})
$$ of length
 $n_{2j+1}$ such that for each $i\le n_{2j+1}/2$, $x_{2i-1}\in \langle e_n\rangle_M$
 and $x_{2i}\in \langle y_k\rangle_k$.
 \end{lemma}
\begin{proof}
Let $j_{1}\in\mathbb{N}$, $j_{1}$ even such that $m^{1/2}_{j_{1}}>n_{2j+1}$. We set
$$
x_{1}=\frac{1}{n_{j_{1}}}\sum_{i=1}^{n_{j_1}}e_{1,i}\,\,\,\text{and}
\,\,\,\,\,
f_{1}=\frac{1}{m_{j_{1}}}\sum_{i=1}^{n_{j_1}}e^{*}_{1,i}\,,
$$
such that $x_{1}\in\langle e_{n}\rangle_{M}$. Let
$j_{2}=\sigma(x_{1},f_{1})$. Using Proposition~\ref{p18} we choose
an $(3,\frac{1}{n_{j_{2}}})$ R.I.S,
$(x_{l}^{2})_{l=1}^{n_{j_{2}}}\in \langle y_{k}\rangle_{k}$ such
that $x_{1}<x^{2}_{l}$ for every $l\leq n_{j_{2}}$. Next we choose
for every $l\leq n_{j_{2}}$ a functional $f^{2}_{l}\in K$ such
that $f^{2}_{l}(x^{2}_{l})\geq
\frac{2}{3}\norm[x^{2}_{l}]\geq\frac{2}{3}$ and
$\text{range}(f^{2}_{l})\subset\text{range}(x^{2}_{l})$. We set
$$
f_{2}=\frac{1}{m_{j_{2}}}\sum_{l=1}^{n_{j_{2}}}f^{2}_{l}\,\,\,
\text{and}\,\,\,\
x_{2}=\frac{c_{2}}{n_{j_{2}}}\sum_{l=1}^{n_{j_2}}x^{2}_{l}
\,\,\,\,\,\text{where}\,\,\, c_{2}=
\frac{1}{6}(1-\frac{m_{j_{2}}}{n^{2}_{j_{2}}})\,\,.
$$
From Proposition~\ref{ris}, it follows that $\norm[x_{2}]\leq
(\frac{1}{m_{j_{2}}}-\frac{1}{n^{2}_{j_{2}}})$. We also have that
\begin{equation}\label{l412e1}
f_{2}(x_{2})\geq
\frac{1}{m_{j_{2}}}\frac{c_2}{n_{j_2}}
\sum_{l=1}^{n_{j_{2}}}f^{2}_{l}(x^{2}_{l})
\geq\frac{2}{3}\frac{c_{2}}{m_{j_{2}}}\geq\frac{1}{12m_{j_{2}}}
\,\,.
\end{equation}
We choose $y_{2}\in\mathbf{Q}$, that is $y_{2}$ is a finite
sequence with  rational coordinates, such that
$\norm[y_{2}-x_{2}]\leq\frac{1}{n^{2}_{j_{2}}}$ and
$\text{supp}(y_{2})=\text{supp}(x_{2})$. It follows that
$\norm[y_{2}]\leq \frac{1}{m_{j_2}}$ and therefore
$(x_{1},f_{1},y_{2},f_{2})$ is a special sequence of length $2$.

We set $j_{3}=\sigma(x_{1},f_{1},y_{2},f_{2})$ and we choose $$
x_{3}=\frac{1}{n_{j_{3}}}\sum_{l=1}^{n_{j_{3}}}e_{3,l}\,\,\,
\text{and}\,\,\,
f_{3}=\frac{1}{m_{j_{3}}}\sum_{l=1}^{n_{j_{3}}}e^{*}_{3,l}\,,
$$
such that
$\text{range}(y_{2})\cup\text{range}(f_{2})<\text{range}(x_{3})$
and $x_{3}\in\langle e_{n}\rangle_{M}$. Next we choose
$x_{4},f_{4}$ and $y_{4}$ as in the second step, and it is clear
that the procedure goes through up to the choice of $x_{n_{2j+1}},
f_{n_{2j+1}}$ and $y_{n_{2j+1}}$.
\end{proof}
\begin{remark}\label{413}
a) Let us observe that the proof of Lemma~\ref{412} yields that if
$\chi=(x_1,f_1,x_2,f_2,\ldots,x_{n_{2j+1}},f_{n_{2j+1}})$ is a
depended sequence, then for every $i\leq n_{2j+1}/2$ it holds that
$x_{2i}=\frac{c_{2i}}{n_{j_{2i}}}\sum_{l=1}^{n_{j_{2i}}}x_{l}^{2i}$,
where $(x^{2i}_{l})_{l}$ is a $(3,n_{j_{2i}})-R.I.S.$,
$j_{2i}=\sigma(\phi_{2i-1})$ and $c_{2i}\leq \frac{1}{6}$. It
follows from Proposition~\ref{ris}  that
\begin{center}
$\norm[m_{j_{2i}}x_{2i}]\leq 1$, and also if $f\in K$ and
$w(f)<m_{j_{2i}}$ then, $f(m_{j_{2i}}x_{2i})\leq\frac{2}{w(f)}$.
\end{center}

b) Definition \ref{depseq} essentially describes that a depended sequence is a small perturbation  of a special sequence.
Its necessity occurs from the restriction in the definition of
the special sequence $\phi=(x_1,f_1,\ldots, x_{n_{2j+1}},f_{n_{2j+1}})$
that each $x_i\in \mathbf{Q}$ (i.e. $x_i(n)$ is a rational number)
not permitting to find such elements $x_i$ in every block subspace.
 \end{remark}

 Next we state the basic estimations of averages related to depended sequences.
\begin{lemma}\label{depest} Let
$\chi=(x_1,f_1,x_2,f_2,\ldots,x_{n_{2j+1}},f_{n_{2j+1}})$
be a depended sequence of length  $n_{2j+1}$.
Then the following inequality holds:
$$ \norm[\frac{1}{n_{2j+1}}\sum\limits_{i=1}^{n_{2j+1}}
  (-1)^{i+1}m_{j_i}x_i] \leq \frac{8}{m_{2j+1}^3}
$$
where $m_{j_i}=w(f_i)$.
\end{lemma}
\begin{lemma}\label{ld}
Let $\phi=(x_{1},f_{1},\ldots,x_{n_{2j+1}},f_{n_{2j+1}})$ be a
\textit{special sequence}. For every $i\leq n_{2j+1}/2$, let
$\sigma(x_{1},f_{1},\ldots,x_{2i-1},f_{2i-1})=j_{2i}$ and let
$y_{2i}=
\frac{m_{j_{2i}}}{n_{j_{2i}}} \sum\limits_{l=1}^{n_{j_{2i}}}e_{k_{l}}$ be
such that
$$
\text{supp}(f_{2i})\cap\text{supp}(y_{2i})=
\emptyset\,\,\,\text{and}\,\,\,
\text{supp}(f_{2i-1})<\text{supp}(y_{2i})<\text{supp}(f_{2i+1})\,.
$$
Then it holds that
$$
\norm[\frac{1}{n_{2j+1}}\sum_{i=1}^{n_{2j+1}/2}y_{2i}] \leq
\frac{8}{m_{2j+1}^{3}}\,\,.
$$
\end{lemma}

These two lemmas are the key ingredients for proving the main
results for the structure of $X_{ius}$ and $\mathcal{B}(X_{ius})$.
We proceed with the proof of the main results and we will provide
the proof of the two lemmas at the end.

\begin{proposition}\label{pd}
Let $M\in [\mathbb{N}]$ and let $(y_{k})_{k}$ be a normalized block
sequence. Then we have that
$$
\text{dist}(S_{\langle e_{n}\rangle_{M}}, S_{\langle
y_{k}\rangle_k})=0\,.
$$
\end{proposition}
\begin{proof}
For a given $\varepsilon>0$ we choose $j\in\mathbb{N}$
such that $\frac{8}{m_{2j+1}^{2}}<\varepsilon$.
 From Lemma \ref{412} there exists a depended sequence
$\chi=(x_{1},f_{1},\ldots,x_{n_{2j+1}},f_{n_{2j+1}})$
such that
$x_{2i-1}\in \langle e_{n}\rangle_{M}$, $x_{2i}\in
\langle y_{k}\rangle_{k}$ for every $i\leq n_{2j+1}/2$.
Set
$$
e=\frac{m_{2j+1}}{n_{2j+1}}\sum\limits_{i=1}^{n_{2j+1}/2}
m_{j_{2i-1}}x_{2i-1}\;\text{ and }\;
y=\frac{m_{2j+1}}{n_{2j+1}}\sum\limits_{i=1}^{n_{2j+1}/2}
m_{j_{2i}}x_{2i}\,.
$$
We have that $e\in\langle e_{n}:n\in M\rangle$
and $y\in\langle y_{i}:i\in M\rangle$.
 From Lemma~\ref{depest} we have that  $\norm[e-y]\leq
\frac{8}{m_{2j+1}^{2}}$.
To obtain a lower estimation of the norm of $e$
and $y$ we consider the functional
$f=\frac{1}{m_{2j+1}}\sum\limits_{i=1}^{n_{2j+1}/2}
\lambda_{f_{2i}}f_{2i-1}+ f_{2i}$ where
$\lambda_{f_{2i}}=f_{2i}(m_{j_{2i}}y_{2i})$
and $\phi=(x_1,f_1,y_2,f_2,\ldots,y_{n_{2j+1}},f_{n_{2j+1}})$
is the special sequence associated to the depended sequence $\chi$.
 From the definition of the depended sequence, $f_{2i}(m_{j_{2i}}x_{2i})\geq\frac{1}{12}$, and
$\norm[x_{2i}-y_{2i}]\leq\frac{1}{n_{j_{2i}}^2}$ for every
$i\leq n_{2j+1}/2$.
It follows that
$$
\lambda_{f_{2i}}=f(m_{j_{2i}}y_{2i})\geq f(m_{j_{2i}}x_{2i})-m_{j_{2i}}\norm[x_{2i}-y_{2i}]
>\frac{1}{12}-\frac{1}{m^{2}_{j_{2i}}} >\frac{1}{24}\,.
$$

Therefore
\begin{equation}\label{ll1}
\Vert e\Vert\geq f(e)=\frac{m_{2j+1}}{m_{2j+1}}
\sum_{i=1}^{n_{2j+1}/2}\frac{
\lambda_{f_{2i}}f_{2i-1}(m_{j_{2i-1}}x_{2i-1})}{n_{2j+1}}
\geq\frac{1}{48}\,,
\end{equation}
and
\begin{equation}
\Vert y\Vert \geq f(y)=\frac{m_{2j+1}}{m_{2j+1}}
\sum_{i=1}^{n_{2j+1}/2}\frac{
f_{2i}(m_{j_{2i}}x_{2i})}{n_{2j+1}}\geq\frac{1}{24}\,. \label{ll2}
\end{equation}

These lower estimations and the fact that
 $\norm[e-y]\leq \frac{8}{m_{2j+1}^{2}}$ easily yields the desired result.
\end{proof}

\begin{lemma}\label{lemmaT}
 Let $T:X_{ius}\to X_{ius}$ be a bounded operator. Then
$$
\lim_{n}\text{dist}(Te_{n},\mathbb{R}e_{n})=0\,.
$$
\end{lemma}
\begin{proof}
Without loss of generality we may assume that $\norm[T]=1$. Since
$(e_{n})$ is weakly null, by a small perturbation of  $T$ we may
assume that $T(e_{n})$ is a finite block, $T(e_n)\in\mathbf{Q}$
and
$\min\text{supp}T(e_{n}){}_{\overset{\longrightarrow}{n}}\infty$.
Let $I(e_{n})$ be the smallest interval  containing
$\text{supp}T(e_{n})\cup\text{supp}(e_{n})$. Passing to a
subsequence $(e_{n})_{n\in M}$, we may assume that
$I(e_{n})<I(e_{m})$ for every $n,m\in M$ with $n<m$.

If the result is not true, we may assume, on passing to a further
subsequence, that there exists $\delta>0$ such that $$
\text{dist}(Te_{n},\mathbb{R}e_{n})>2\delta\,\,\,\,\text{for
every}\,n\in M\,. $$ It follows that $\Vert
P_{n-1}Te_{n}\Vert>\delta$ or $\Vert (I-P_{n})Te_{n}\Vert>
\delta$. Therefore for every $n\in M$ we can choose $x_{n}^{*}\in
K$ such that
\begin{equation}\label{eq}
x_{n}^{*}(Te_{n})\geq \delta,\,\,\,
\text{range}(x_{n}^{*})\cap\text{range}(e_{n})
 =\emptyset,\,\,\text{
and}\,\,\, \text{range}(x_{n}^{*})\subset I(e_{n})\,.
\end{equation}
Since $T$ is bounded, for every $j\in\mathbb{N}$ we have that
$$
\norm[T(\frac{1}{n_{2j}}\sum_{i=1}^{n_{2j}}e_{k_{i}})]\leq
\norm[T]\,\norm[\frac{1}{n_{2j}}\sum_{i=1}^{n_{2j}}e_{k_{i}}]=
\frac{1}{m_{2j}}\,.
$$
Also for every $j\in\mathbb{N}$ and $k_1<k_2<\cdots<k_{n_{2j}}$ in
$M$, the functional
$h_{2j}=\frac{1}{m_{2j}}\sum\limits_{i=1}^{n_{2j}}x^{*}_{k_{i}}$
is in $K$ and
$$
\norm[T(\frac{1}{n_{2j}}\sum_{i=1}^{n_{2j}}e_{k_{i}})]=
\norm[\frac{1}{n_{2j}}\sum_{i=1}^{n_{2j}}Te_{k_{i}}]\geq
h_{2j}(\frac{1}{n_{2j}}\sum_{i=1}^{n_{2j}}Te_{k_{i}})
\geq\frac{\delta}{m_{2j}}\,\,.
$$
We consider now a \textit{special sequence}
$\phi=(x_{1},f_{1},\ldots,x_{n_{2j+1}},f_{n_{2j+1}})$ which is
defined as follows: for every $i\geq 0$,
\begin{align*}
&x_{2i+1}= \frac{1}{n_{\sigma(\phi_{2i})}}
\sum_{j=1}^{n_{\sigma(\phi_{2i})}}e_{2i+1,j}\,,\quad &
f_{2i+1}=\frac{1}{m_{\sigma(\phi_{2i})}}
\sum_{j=1}^{n_{\sigma(\phi_{2i})}}e_{2i+1,j}^{*}\\
&x_{2i}= \frac{1}{n_{\sigma(\phi_{2i-1})}}
\sum_{j=1}^{n_{\sigma(\phi_{2i-1})}}Te_{2i,j}\,,
&f_{2i}=\frac{1}{m_{\sigma(\phi_{2i-1})}}
\sum_{j=1}^{n_{\sigma(\phi_{2i-1})}}x^{*}_{2i,j}
\end{align*}
where $e_{i,\ell}\in \{e_n:\;n\in M\}$,
 $x^{*}_{2i,j}$, $Te_{2i,j}$ satisfies \eqref{eq}, and
$I(e_{i,\ell})<I(e_{s,j})$ if either $i<s$  or $i=s$ and $\ell<j$.
This is possible by our assumption $I(e_{n})<I(e_{m})$ for
$n,m\in M$ with $n<m$. Observe that
$f_{2i}(m_{\sigma(\phi_{2i-1})}x_{2i})\geq \delta$ and also that
$\text{range}(f_{\ell})\cap\text{range}(x_{2i})=\emptyset$ for
every $\ell\not=2i$. Consider now the following vector:
$$
x=\frac{1}{n_{2j+1}}\sum_{i=1}^{n_{2j+1}/2}\frac{m_{\sigma
(\phi_{2i-1})}}{n_{\sigma(\phi_{2i-1})}}
\sum_{j=1}^{n_{\sigma(\phi_{2i-1})}}e_{2i,j}\,\,.
$$
Then
$$
T(x)=\frac{1}{n_{2j+1}}\sum_{i=1}^{n_{2j+1}/2}m_{\sigma
(\phi_{2i-1})}x_{2i}\,\,,
$$
and
$$
\norm[Tx]\geq\frac{1}{m_{2j+1}}
\sum_{i=1}^{n_{2j+1}/2}(\lambda_{f_{2i}}f_{2i-1}+f_{2i})
Tx\geq\frac{\delta}{2m_{2j+1}}\,\,.
$$
On the other hand, if $y_{2i}=\frac{m_{\sigma
(\phi_{2i-1})}}{n_{\sigma(\phi_{2i-1})}}
 \sum\limits_{j=1}^{n_{\sigma(\phi_{2i-1})}}e_{2i,j}
$, then we have that
 $\text{supp}(y_{2i})
\cap\text{supp}f_{2i}=\emptyset$ and $x_{2i-1}<y_{2i}<x_{2i+1}$
for every $i\leq n_{2j+1}/2$, and therefore by Lemma~\ref{ld} we
have that
$$
\norm[x]= \norm[\frac{1}{n_{2j+1}}\sum_{i=1}^{n_{2j+1}/2}y_{2i}]
\leq\frac{8}{m_{2j+1}^{3}}\,.
$$
It follows that
$\norm[T]\geq \frac{\delta}{16} m^{2}_{2j+1}$, a contradiction for
$j$ sufficiently large.
\end{proof}
\begin{proposition}\label{pss}
Let $T:X_{ius}\to X_{ius}$ be a bounded operator. Then there
exists $\lambda\in\mathbb{R}$ such that $T-\lambda I$ is strictly
singular.
\end{proposition}
\begin{proof}
By Lemma~\ref{lemmaT} there exists $\lambda\in\mathbb{R}$ and
$M\in [\mathbb{N}]$ such that $\lim_{n\in M}\norm[Te_{n}-\lambda
e_{n}]=0$. Let $\varepsilon>0$. Passing to a further subsequence $(e_{n_{k}})_{k}$,
we may assume that  $\norm[Te_{n_{k}}-\lambda e_{n_{k}}]\leq \varepsilon
2^{-k}$ for every $k\in\mathbb{N}$. It follows that the restriction of
$T-\lambda I$ to $[e_{n_{k}}, k\in\mathbb{N}]$ is of norm less than
$\varepsilon$. By Proposition~\ref{pd} it follows that $T-\lambda
I$ is strictly singular.
\end{proof}
The following two corollaries are consequences of
Proposition~\ref{pss} (see \cite{GM}).
\begin{corollary}
There does not  exist a non trivial projection $P:X_{ius}\to
X_{ius}$.
\end{corollary}
\begin{corollary}
The space $X_{ius}$ is not isomorphic to any proper subspace of
it.
\end{corollary}

It remains to prove lemmas  \ref{depest} and \ref{ld}.
We start with the following.
\begin{lemma}\label{pd1}
Let $j\in\mathbb{N}$,
$n_{2j+1}<m_{j_{1}}<m_{j_{2}}<\ldots<m_{j_{2r}}$ be such that
$2r\leq n_{2j+1}<m_{j_{1}}^{1/2}.$ Let also $j_{0}\in\mathbb{N}$
be such that $m_{j_{0}}\not= m_{j_{i}}$ for every $i=1,\ldots,2r$
and $m^{1/2}_{j_{0}}>n_{2j+1}$. Then if $h_{1}<\ldots<h_{2r}\in K$
are such that $w(h_{i})=m_{j_{i}}$ for every $i=1,\ldots,2r$, then

a)
\begin{equation}\label{e1pd1}
\vert(\sum_{k=1}^{r}\lambda_{2k-1}h_{2k-1}+h_{2k})
(\frac{m_{j_{0}}}{n_{j_{0}}}\sum_{l=1}^{n_{j_{0}}}e_{k_l})\vert
<\frac{1}{n_{2j+1}}\,,
\end{equation}
for every choice of real numbers $(\lambda_{2k-1})_{k=1}^{r}$
with $\vert\lambda_{2k-1}\vert\leq 1$ for every $k\leq r$.

b) If $(x_{l})_{l=1}^{n_{j_{0}}}$ is a
$(3,\frac{1}{n_{j_{0}}})-$R.I.S of $\ell_{1}$ averages, then
\begin{equation}\label{e2pd1}
\vert(\sum_{k=1}^{r}\lambda_{2k-1}h_{2k-1}+h_{2k})
(\frac{m_{j_{0}}}{n_{j_{0}}}\sum_{l=1}^{n_{j_{0}}}x_{l})\vert \leq
\frac{1}{n_{2j+1}}\,,
\end{equation}
for every choice of real numbers $(\lambda_{2k-1})_{k=1}^{r}$ with
$\vert\lambda_{2k-1}\vert\leq 1$ for every $k\leq r$.
\end{lemma}
\begin{proof}
We shall give the proof of b) and we shall indicate the minor
changes for the proof of a).

From the estimations on the R.I.S, Proposition~\ref{ris}, for
every $k\leq 2r$ we have that
\begin{equation}\label{e4pd1}
\vert h_{k}(\frac{m_{j_{0}}}{n_{j_{0}}}\sum_{l=1}^{n_{j_{0}}}x_{l})
\vert\leq
\begin{cases}
\frac{9}{w(h_{k})},\quad &\text{if}\,\, w(h_{k})<m_{j_{0}}\\
\frac{3}{m_{r}}+\frac{6}{n_{j_{0}}},\,\,&\text{if}\,\,
w(h_{k})=m_{r}>m_{j_{0}}\,.
\end{cases}
\end{equation}
Using that $m_{j+1}=m^{5}_{j}$ for every $j$ and
$\vert\lambda_{2k-1}\vert\leq 1$ for every $k\leq r$ and
\eqref{e4pd1}, we get that
\begin{align*}
\vert(\sum_{k=1}^{r}\lambda_{2k-1}h_{2k-1}+h_{2k})
(\frac{m_{j_{0}}}{n_{j_{0}}}\sum_{l=1}^{n_{j_{0}}}x_{l})\vert
&\leq \sum_{k:w(h_{k})<m_{j_{0}}}\frac{9}{w(h_{k})}+
\sum_{r>j_{0}}\frac{3}{m_{r}}+\frac{12r}{n_{j_{0}}}
\\
&\leq
\frac{10}{w(h_{1})}+\frac{4}{m^{2}_{j_{0}}}+\frac{12r}{n_{j_{0}}}
<\frac{1}{n_{2j+1}}\,.
\end{align*}

For the proof of a)
using Lemma~\ref{42}, for the estimations on the basis
we get  the corresponding inequality to \eqref{e4pd1}, from which
 follows inequality \eqref{e1pd1}.
\end{proof}

\begin{proof}[Proof of Lemma \ref{depest}]
Let $\chi=(x_1,f_1,\ldots,x_{n_{2j+1}},f_{n_{2j+1}})$ be a
depended sequence and
$\phi=(y_1,f_1,y_2,f_2,\ldots,y_{n_{2j+1}},f_{n_{2j+1}})$ the
special sequence associated to $\chi$. In the rest of the proof we
shall assume that $\chi=\phi$. The general proof follows by slight and obvious modifications of the present proof. Hence we assume
that $\phi=(x_1,f_1,\ldots,x_{n_{2j+1}},f_{n_{2j+1}})$.

From Lemma~\ref{42} and Remark~\ref{413}a) it follows  that the
sequence $(m_{j_{i}}x_{i})_{i=1}^{n_{2j+1}}$ satisfies assumptions a), c) of the basic
inequality for  $C=2$. Furthermore the properties of the function
$\sigma$ yield that assumption  b) is also satisfied for
$\varepsilon=1/n_{2j+1}$.

The rest of the proof is devoted to
establish that the sequence $(m_{j_{i}}x_{i})_{i}$ satisfies the
crucial condition d) for $m_{j_{0}}=m_{2j+1}$ and
$(b_{i})_{i}=(\frac{(-1)^{i+1}}{n_{2j+1}})_{i}$.

First we consider  $f\in K_{\phi}$. Then  $f$ is of the form
$$
f=E(\frac{\varepsilon}{m_{2j+1}}
(\lambda_{f^{\prime}_{2}}f_{1}+f^{\prime}_{2}+\ldots +
\lambda_{f^{\prime}_{n_{2j+1}}}f_{n_{2j+1}-1}+
f^{\prime}_{n_{2j+1}})\,)\,,
$$
where $\varepsilon\in\{-1,1\}$ and $E$ an interval of
$\mathbb{N}$. Let us recall that $w(f^{\prime}_{2i})=w(f_{2i})$
and $\text{supp}(f^{\prime}_{2i})=\text{supp}(f_{2i})$  and
therefore
$\text{range}(f^{\prime}_{2i})\cap\text{range}(x_{k})=\emptyset$
for every $k\not=2i$. Let
$$
i_{0}=\min\{i\leq n_{2j+1}/2: \text{supp}(f)\cap
(\text{range}(x_{2i-1})\cup\text{range}(x_{2i}))\not=\emptyset\}\,.
$$
Then
\begin{align}
\vert f(\sum_{i=1}^{n_{2j+1}}(-1)^{i+1}m_{j_{i}}x_{i})\vert =&
\vert E\frac{1}{m_{2j+1}}\sum_{k=1}^{n_{2j+1}/2}
 (\lambda_{f^{\prime}_{2k}}f_{2k-1}+f^{\prime}_{2k})
 (\sum_{i=1}^{n_{2j+1}}
 (-1)^{i+1}m_{j_{i}}x_{i})\vert
 \leq \notag
\\
& \frac{1}{m_{2j+1}}
\vert\lambda_{f^{\prime}_{2i_{0}}} Ef_{2i_{0}-1}(m_{j_{2i_{0}-1}}x_{2i_{0}-1})-
 Ef^{\prime}_{2i_{0}}(m_{j_{2i_{0}}}x_{2i_{0}})\vert
 \label{c2}
\\
+&\frac{1}{m_{2j+1}}\vert\sum_{i=i_{0}+1}^{n_{2j+1}/2}
 (\lambda_{f^{\prime}_{2i}}f_{2i-1}(m_{j_{2i-1}}x_{2i-1})-
 f^{\prime}_{2i}(m_{j_{2i}}x_{2i}))\vert\,.
\label{c3}
\end{align}
To estimate the sum in \eqref{c2} and \eqref{c3}, we partition the
set $\{i_{0},\ldots,n_{2j+1}/2 \}$ into two sets $A$ and $B$,
where $A=\{i: f^{\prime}_{2i}(x_{2i})\not= 0\}$ and $B$ is its
complement. For every $i\in A$, $i>i_{0}$, using that
$\lambda_{f^{\prime}_{2i}}=f^{\prime}_{2i}(m_{j_{2i}}x_{2i})$, we
have that
\begin{equation}\label{c4}
\lambda_{f^{\prime}_{2i}}f_{2i-1}(m_{j_{2i-1}}x_{2i-1})-
f^{\prime}_{2i}(m_{j_{2i}}x_{2i})=
f^{\prime}_{2i}(m_{j_{2i}}x_{2i})-
f^{\prime}_{2i}(m_{j_{2i}}x_{2i}) =0\,.
\end{equation}
For every $i\in B$ we have that $f^{\prime}_{2i}(x_{2i})=0$, and
therefore, $|\lambda_{f^{\prime}_{2i}}|=\frac{1}{n_{2j+1}^{2}}$,
see \eqref{ek0}. It follows that, for every $i\in B$, $i>i_{0}$
\begin{equation}\label{c5}
\vert \lambda_{f^{\prime}_{2i}}f_{2i-1}(m_{j_{2i-1}}x_{2i-1})-
f^{\prime}_{2i}(m_{j_{2i}}x_{2i})\vert =\vert
\lambda_{f^{\prime}_{2i}}\vert =\frac{1}{n^{2}_{2j+1}}\,\,.
\end{equation}
For the sum $\vert
\lambda_{f^{\prime}_{2i_{0}}} Ef_{2i_{0}-1}(m_{j_{2i_{0}-1}}x_{2i_{0}-1})
-  Ef^{\prime}_{2i_{0}}(m_{j_{2i_{0}}}x_{2i_{0}}) \vert$
distinguishing whether or not $Ef_{2i_{0}-1}=0$ and whether
$i_{0}\in A$ or $i_{0}\in B$, it follows easily using the previous
arguments that
\begin{equation}\label{c6}
\vert
\lambda_{f^{\prime}_{2i_{0}}} Ef_{2i_{0}-1}(m_{j_{2i_{0}-1}}x_{2i_{0}-1})
-Ef^{\prime}_{2i_{0}}( m_{j_{2i_{0}}}x_{2i_{0}}) \vert\leq 1
\end{equation}
Summing up \eqref{c4}-\eqref{c6} we have that
\begin{equation}\label{c8}
\vert
f(\frac{1}{n_{2j+1}}
\sum_{i=1}^{n_{2j+1}}(-1)^{i+1}m_{j_{i}}x_{i})\vert\leq
\frac{1}{m_{2j+1}}(\frac{1}{n_{2j+1}}+\frac{1}{n^{2}_{2j+1}})<
\frac{1}{n_{2j+1}}\,.
\end{equation}

Consider now a \textit{special sequence}
$\psi=(y_{1},g_{1},y_{2},g_{2},\ldots,
y_{n_{2j+1}},g_{n_{2j+1}})$. Let $i_{1}=
\min\{i\in\{1,\ldots,n_{2j+1}\}
: y_{i}\not=x_{i}\,\,\text{or}\,\,
g_{i}\not=f_{i}\}$, and $k_{0}\in\mathbb{N}$ such that
$i_{1}=2k_{0}-1$ or $2k_{0}$.

Consider a functional $g\in K_{\psi}$ which is defined from this
special sequence. Then we have that
$$
g=E(\frac{1}{m_{2j+1}}
(\lambda_{g^{\prime}_{2}}g_{1}+g^{\prime}_{2}+ \ldots+
\lambda_{g^{\prime}_{n_{2j+1}}}g_{n_{2j+1}-1}+g^{\prime}_{n_{2j+1}})\,,
$$
where $E$ is an interval of $\mathbb{N}$ and
$w(g^{\prime}_{2i})=w(g_{2i})$ for every $i\leq n_{2j+1}/2$.
Observe that
$\text{range}(x_{i})\cap\text{range}(g_{k})=\emptyset$ for every
$i\geq i_{1}$ and every $k<i_{1}$. Let
$$
i_{0}=\min\{i\leq n_{2j+1}/2: \text{supp}(g)\cap
(\text{range}(x_{2i-1})\cup\text{range}(x_{2i}))\not=\emptyset\}\,.
$$
Let $i_{0}<k_{0}$. Then
\begin{align}
\vert g(\sum_{i=1}^{n_{2j+1}}(-1)^{i+1}m_{j_{i}}x_{i})\vert &\leq
\notag\\
&\frac{1}{m_{2j+1}} \Big( \vert
E\lambda_{g^{\prime}_{2i_{0}}} g_{2i_{0}-1}(m_{j_{2i_{0}-1}}x_{2i_{0}-1})-
Eg^{\prime}_{2i_{0}}(m_{j_{2i_{0}}}x_{2i_{0}}) \vert
 \label{re}
\\
&\qquad\qquad+ \vert \sum_{i=i_{0}+1}^{k_{0}-1}
(\lambda_{g^{\prime}_{2i}}g_{2i-1}(m_{j_{2i-1}}x_{2i-1})-
g^{\prime}_{2i}(m_{j_{2i}}x_{2i}) \vert\Big) \label{re1}
\\
& +\frac{1}{m_{2j+1}}\vert \sum_{k\geq
k_{0}}(\lambda_{g^{\prime}_{2k}}g_{2k-1}+
g^{\prime}_{2k}) (\sum_{i\geq k_{0}}
m_{j_{2i-1}}x_{2i-1}-m_{j_{2i}}x_{2i})\vert\,. \label{re2}
\end{align}
where the sum in \eqref{re1} makes sense when $i_{0}<k_{0}-1$. If $i_{0}\geq k_{0}$ we get that
\begin{align*}
\vert g(\sum_{i=1}^{n_{2j+1}}(-1)^{i+1}m_{j_{i}}x_{i})\vert &\leq
\frac{1}{m_{2j+1}}\vert E\sum_{k\geq
k_{0}}(\lambda_{g^{\prime}_{2k}}g_{2k-1}+
g^{\prime}_{2k}) (\sum_{i\geq i_{0}}
m_{j_{2i-1}}x_{2i-1}-m_{j_{2i}}x_{2i})\vert\,.
\end{align*}
The proof of the upper estimation for the two cases is almost identical, so we shall give the proof in the case $i_{0}<k_{0}$.

As in the previous case, for the sum in \eqref{re},\eqref{re1} we
have that
\begin{align}
\vert E\lambda_{g^{\prime}_{2i_{0}}}
g_{2i_{0}-1}(m_{j_{2i_{0}-1}}x_{2i_{0}-1})-&
Eg^{\prime}_{2i_{0}}(m_{j_{2i_{0}}}x_{2i_{0}}) \vert +
\notag\\
&\vert \sum_{i=i_{0}+1}^{k_{0}-1}
(\lambda_{g^{\prime}_{2i}}g_{2i-1}(m_{j_{2i-1}}x_{2i-1})-
g^{\prime}_{2i}(m_{j_{2i}}x_{2i}) \vert\leq 2\,.\label{re3}
\end{align}

To estimate  the sum in \eqref{re2}, first we observe that from
the injectivity of $\sigma$ it follows that there exists at most
one $k\geq i_{1}$ such that
$$
 w(g_{k})\in \{m_{j_{i}}: i_{1}\leq i\leq n_{2j+1}\}\,.
$$
Let $2i-1\geq i_{1}$ be such that $m_{j_{2i-1}}\not=
w(g_{k})$ for every $k\geq i_{1}$. Then
functionals $g_{2k-1},g^{\prime}_{2k}$, $k\geq k_{0}$ satisfy
 the assumptions of Lemma~\ref{pd1}, and therefore we get that
\begin{equation}\label{r4}
\vert\sum_{k\geq k_{0}}(\lambda_{g^{\prime}_{2k}}g_{2k-1}+g^{\prime}_{2k})
(m_{j_{2i-1}}x_{2i-1})\vert\leq\frac{1}{n_{2j+1}}\,.
\end{equation}
Also for every $2i\geq i_{1}$ such that $m_{j_{2i}}\not=w(g_{k})$
for every $k\geq i_{1}$, the functionals
$g_{2k-1},g^{\prime}_{2k}$, $k\geq k_{0}$ satisfy the
assumptions of Lemma~\ref{pd1}, and therefore we get that
\begin{equation}\label{r5}
\vert\sum_{k\geq k_{0}}
(\lambda_{g^{\prime}_{2k}}g_{2k-1}+g^{\prime}_{2k})
(m_{j_{2i}}x_{2i})\vert \leq\frac{1}{n_{2j+1}}\,.
\end{equation}
For the unique $i\geq i_{1}$, such that there exists $k\geq i_{1}$
and $w(g_{k})=m_{j_{i}}$, if such an $i$ exists, we have that, using
Lemma~\ref{pd1}
\begin{equation}\label{r6}
\vert\sum_{k\geq k_{0}}(\lambda_{g^{\prime}_{2k}}g_{2k-1}+
g^{\prime}_{2k}) (m_{j_{i}}x_{i})\vert\leq 1+\frac{1}{n_{2j+1}}\,.
\end{equation}
Now we distinguish if $i_{1}=2k_{0}-1$ or $i_{1}=2k_{0}$. If
$i_{1}=2k_{0}-1$, we have that
$\text{range}(g_{k})\cap\text{range}(x_{i})=\emptyset$ for every
$k<2k_{0}-1$ and every $i\geq 2k_{0}-1$, and from
\eqref{r4}-\eqref{r6} we get that
\begin{align}\label{r7}
\vert \sum_{k\geq k_{0}}
(\lambda_{g^{\prime}_{2k}}g_{2k-1}+g^{\prime}_{2k})
\Bigl(\frac{1}{n_{2j+1}}
\sum_{i=2k_{0}-1}^{n_{2j+1}}&(-1)^{i+1}m_{j_{i}}x_{i})\Bigr)
\vert \notag\\
&\leq \frac{1}{n_{2j+1}}(1+\frac{1}{n_{2j+1}}+
\frac{n_{2j+1}}{n_{2j+1}}) <\frac{3}{n_{2j+1}}\,.
\end{align}
If $i_{1}=2k_{0}$ then we have that
$\text{range}(x_{2k_{0}-1})\cap\text{range}(g_{k})=\emptyset$ for
every $k\geq 2k_{0}$ and $k<2k_{0}-1$, and from
\eqref{r4}-\eqref{r6} we get that
\begin{align}\label{r8}
\vert\sum_{k\geq k_{0}}
(\lambda_{g^{\prime}_{2k}}g_{2k-1}+g^{\prime}_{2k})
&\Bigl(\frac{1}{n_{2j+1}}
\sum_{i=2k_{0}-1}^{n_{2j+1}}(-1)^{i+1}m_{j_{i}}x_{i})\Bigr)\vert \\
&\leq \frac{1}{n_{2j+1}}\Bigl(
\vert \lambda_{g^{\prime}_{2k_{0}-1}}g_{2k_{0}-1}
(m_{j_{2k_{0}-1}}x_{2k_{0}-1})\vert \notag \\
&\qquad\qquad+ \vert\sum_{k\geq k_{0}}
(\lambda_{g^{\prime}_{2k}}g_{2k-1}+g^{\prime}_{2k}) (
\sum_{i=2k_{0}}^{n_{2j+1}}(-1)^{i+1}m_{j_{i}}x_{i})\vert\Bigr)
\notag
\\
&\leq \frac{1}{n_{2j+1}}+ \frac{1}{n_{2j+1}}
(1+\frac{1}{n_{2j+1}}+ \frac{n_{2j+1}}{n_{2j+1}})
<\frac{4}{n_{2j+1}}\,.\notag
\end{align}
From  \eqref{re3},\eqref{r7} and \eqref{r8} we get that
\begin{equation}\label{43}
 \vert
g(\frac{1}{n_{2j+1}}
\sum_{i=1}^{n_{2j+1}}(-1)^{i+1}m_{j_{i}}x_{i})\vert \leq
\frac{1}{m_{2j+1}}(\frac{2}{n_{2j+1}}+ \frac{4}{n_{2j+1}})
<\frac{1}{n_{2j+1}}\,.
\end{equation}
The inequalities \eqref{c8} and \eqref{43} yield that indeed
condition d) is satisfied for $\varepsilon=1/n_{2j+1}.$
 Proposition~\ref{ris} (2) derives the
desired result and the proof is complete.
\end{proof}
\begin{proof}[Proof of Lemma \ref{ld}]
To prove this we shall follow similar arguments as in the proof of
Lemma~\ref{depest}. We shall establish conditions a), b), c) and d)
of the basic inequality, for $C=2$,
$\varepsilon=\frac{1}{n_{2j+1}}$ and
 $m_{j_{0}}=m_{2j+1}$.
Lemma~\ref{42}  yields that the sequence $(y_{2i})_{i}$ satisfies the assumptions a) and c) of the basic inequality for $C=2$. Furthermore the properties of the function
$\sigma$ yield that assumption  b) is  also satisfied for
$\varepsilon=1/n_{2j+1}$.

To establish condition d) we shall show that for every $f\in K$
with $w(f)=m_{2j+1}$, it holds that
 $$
\vert
f(\frac{1}{n_{2j+1}}\sum_{i=1}^{n_{2j+1}/2}y_{2i})\vert \leq
\frac{1}{m_{2j+1}}(\frac{1}{n_{2j+1}}+\frac{1}{n_{2j+1}}) <
\frac{1}{n_{2j+1}}\,.
$$
First let us observe that for every $f\in K_{\phi}$,
$f=E\frac{1}{m_{2j+1}}
\sum\limits_{k=1}^{\frac{n_{2j+1}}{2}}
(\lambda_{f^{\prime}_{2k}}f_{2k-1}+f^{\prime}_{2k})$
it holds that $f(\frac{1}{n_{2j+1}}
\sum\limits_{i=1}^{n_{2j+1}/2}y_{2i})=0$.
This is due to
$\text{supp}f_{2i}^{\prime}=\text{supp}f_{2i}$
and $\text{supp}(f_{2i-1})<y_{2i}<\text{supp}(f_{2i+1})$
 for every $i\leq n_{2j+1}/2$.

Let $\phi=(z_{1},g_{1},z_{2},g_{2},\ldots,
z_{n_{2j+1}},g_{n_{2j+1}})$ be a \textit{special sequence} of
length $n_{2j+1}$ and let $f=E\frac{1}{m_{2j+1}}
\sum\limits_{k=1}^{\frac{n_{2j+1}}{2}}
(\lambda_{g^{\prime}_{2k}}g_{2k-1}+g^{\prime}_{2k})$ belonging to
$K_{\phi}$. Without loss of generality we may assume that
 $E=\mathbb{N}$. Let $i_{1}=\min\{i\leq n_{2j+1}: z_{i}\not=
x_{i}\,\,\text{or}\,\,f_{i}\not=g_{i}\}$,
 and $k_{0}\in\mathbb{N}$ such that  $i_{1}=2k_{0}-1$ or
$i_1=2k_{0}$.
Observe that
$\text{range}(g_{k})\cap\text{range}(y_{2i})=\emptyset$ for every
$k<i_{1}$ and every $2i\geq i_{1}$.

From the injectivity of
$\sigma$, it follows that there exists at most one $k\geq i_{1}$
such that
$$ w(g_{k})\in
 \{m_{j_{i}}:i_{1}\leq i \leq n_{2j+1}\}\,.
$$
Let $2i\geq i_{1}$ such that
$w(g_{k})\not=m_{j_{2i}}$ for all $k\geq i_{1}$. Then the
functionals $g_{2k-1}, g^{\prime}_{2k}$, $k\geq k_{0}$ satisfy
the assumptions of Lemma~\ref{pd1}(a), and therefore it follows that
\begin{equation}\label{l2}
\vert(\sum_{k\geq k_{0}}
\lambda_{g^{\prime}_{2k}}g_{2k-1}+g^{\prime}_{2k})(y_{2i}) \vert
 <\frac{1}{n_{2j+1}}\,.
\end{equation}
For the unique $2i\geq i_{1}$ such that there exists $k\geq i_{1}$
with $w(g_{k})=m_{j_{2i}}$, if such $2i$ exists, we have that
\begin{equation}\label{l3}
\vert(\sum_{k\geq k_{0}}
\lambda_{g^{\prime}_{2k}}g_{2k-1}+g^{\prime}_{2k})(y_{2i})
 \vert < 1+\frac{1}{n_{2j+1}}\,.
\end{equation}
Summing up \eqref{l2}-\eqref{l3} we get that
\begin{equation}\label{ee}
\vert f(\frac{1}{n_{2j+1}}\sum_{i=1}^{n_{2j+1}/2}y_{2i})\vert \leq
\frac{1}{m_{2j+1}}(\frac{1}{n_{2j+1}}+\frac{1}{n_{2j+1}})
<\frac{1}{n_{2j+1}}\,.
\end{equation}
Inequality \eqref{ee} implies that condition d) of the basic
inequality  is fulfilled, and  Proposition~\ref{ris} yields the
desired result.
\end{proof}

 \end{document}